\documentclass[reqno,11pt]{amsart}
\usepackage{amsmath,amsfonts,amscd,amssymb,latexsym}

\usepackage{hyperref}
\usepackage{soul}
\usepackage{mathrsfs}  

\usepackage[nobysame]{amsrefs}
\usepackage{comment}
\usepackage{soul}
\usepackage{tikz}
\usepackage{marginnote}
\usepackage{anysize}
\usepackage{bbm}
\usepackage{cancel}
\usepackage[normalem]{ulem}
\usepackage[percent]{overpic}

\numberwithin{equation}{section}

\theoremstyle{plain}
\newtheorem{theorem}[subsection]{Theorem}
\newtheorem{lemma}[subsection]{Lemma}
\newtheorem{corollary}[subsection]{Corollary}
\newtheorem{proposition}[subsection]{Proposition}

\theoremstyle{definition}
\newtheorem{definition}[subsection]{Definition}

\theoremstyle{remark}
\newtheorem{remark}[subsection]{Remark}

\newcommand{\ind}{\operatorname{ind}}
\newcommand{\id}{\operatorname{id}}
\newcommand{\dom}{\operatorname{dom}}

\newcommand{\dm}{\partial M}

\newcommand{\RR}{\mathbb{R}}

\newcommand{\ZZ}{\mathbb{Z}}

\newcommand{\AAA}{\mathcal{A}}

\newcommand{\DD}{\mathcal{D}}

\newcommand{\bfu}{\mathbf{u}}

\newcommand{\upper}{\uppercase\expandafter}

\newcommand{\n}{\nabla}
\newcommand{\p}{\partial}
\newcommand{\pM}{{\p M}}

\newcommand{\Hom}{\operatorname{Hom}}
\newcommand{\End}{\operatorname{End}}
\newcommand{\ad}{{\rm ad}}

\newcommand{\oB}{\bar{B}}
\newcommand{\spf}{\operatorname{sf}}
\newcommand{\PP}{\mathbb{P}}
\renewcommand{\AA}{\mathbb{A}}

\newcommand{\tilM}{\tilde M}

\newcommand{\AS}{\alpha_{\rm AS}}
\newcommand{\oeta}{\bar{\eta}}

\begin{document}

\author{Maxim Braverman${}^\dag$}
\address{Department of Mathematics,
Northeastern University,
Boston, MA 02115,
USA}

\email{maximbraverman@neu.edu}
\urladdr{www.math.neu.edu/~braverman/}

\author{Pengshuai Shi}
\address{Beijing International Center for Mathematical Research (BICMR),
Beijing (Peking) University,
Beijing 100871,
China}

\email{pengshuai.shi@gmail.com}

\thanks{${}^\dag$†Partially supported by the Simons Foundation collaboration grant \#G00005104.}

\title[APS index with non-compact boundary]{An APS index theorem for even-dimensional manifolds with non-compact boundary}

\begin{abstract}
We study the index of  the APS boundary value problem for a strongly Callias-type operator $\DD$ on a complete Riemannian manifold $M$. We use this index to define  
the relative $\eta$-invariant $\eta(\AAA_1,\AAA_0)$ of two strongly Callias-type operators, which are equal outside of a compact set. Even though in our situation the $\eta$-invariants of $\AAA_1$ and $\AAA_0$ are not defined, the relative $\eta$-invariant behaves as if it were the difference $\eta(\AAA_1)-\eta(\AAA_0)$. We also define the spectral flow of a family of such operators and use it to compute the variation of the relative $\eta$-invariant. 
\end{abstract}

\maketitle

\setcounter{tocdepth}{1}

\section{Introduction}\label{S:intro}

In \cite{BrShi17} we studied the index of the Atiyah--Patodi--Singer (APS) boundary value problem for a strongly  Callias-type operator on a complete {\em odd-dimensional} manifold with non-compact boundary. We used this index to define the {\em relative $\eta$-invariant} $\eta(\AAA_1,\AAA_0)$ of two strongly Callias-type operators  on {\em even-dimensional} manifolds, assuming that $\AAA_0$ and $\AAA_1$ coincide outside of a compact set. 

In this paper we discuss an even-dimensional analogue of \cite{BrShi17}. Many parts of the paper are parallel to the discussion in \cite{BrShi17}. However, there are two important differences. First, the Atiyah--Singer integrand was, of course, equal to 0 in \cite{BrShi17}, which simplified many formulas. In particular, the relative $\eta$-invariant was an integer. As opposed to it,  in the current paper the Atiyah--Singer integrand plays an important role and the relative $\eta$-invariant is a real number. More significantly, the proof of the main result in \cite{BrShi17} was based on the application of the Callias index theorem, \cite{Anghel93Callias,Bunke95}. This theorem is not available in our current setting. Consequently, a completely different proof is proposed in Section~\ref{S:esscyl}.

Our results provide a new tool to study anomalies in quantum field theory. Mathematical description of many anomalies is given by index theorems for boundary value problems, cf. \cite{Witten85GravAnom,AtSinger84,Freed86,BarStrohmaier16}, \cite[Ch.~11]{Bertlmann01}. However, most mathematically rigorous descriptions of anomalies in the literature only work on compact manifolds. The results of the current paper allow to extend many of these descriptions to non-compact setting, thus providing a mathematically rigorous description of anomalies in more realistic physical situations. In particular, B\"ar and Strohmaier, in \cite{BarStrohmaier15,BarStrohmaier16}, gave a mathematically rigorous description of chiral anomaly  by considering an APS boundary problem for Dirac operator on a Lorentzian spatially compact manifold. In a recent preprint \cite{Br-Lorentz} the first author extended the results of \cite{BarStrohmaier15} to spatially non-compact case. The analysis in \cite{Br-Lorentz} depends heavily on the results of the current paper. Another applications to the anomaly considered in \cite{HoravaWitten96} will appear in \cite{BrShi18}.

We now briefly describe our main results.

\subsection{Strongly Callias-type operators}\label{SS:IstrCallias}

A Callias-type operator on a complete Riemannian manifold $M$ is an operator of the form $\DD= D+\Psi$ where $D$ is a Dirac operator and $\Psi$ is a self-adjoint potential which anticommutes with the Clifford multiplication and satisfies certain growth conditions at infinity. In this paper we impose  slightly stronger growth conditions on $\Psi$ and refer to the obtained operator $\DD$ as a {\em strongly Callias-type operator}. Our conditions on the growth of $\Psi$ guarantee that the spectrum of $\DD$ is discrete.

The Callias-type index theorem, proven in different forms in  \cite{Callias78,BottSeeley78,BruningMoscovici,Anghel93Callias,Bunke95}, computes the index of a Callias-type operator on a complete {\em odd-dimensional} manifold as the index of a certain operator induced by $\DD$ on a compact hypersurface. Several generalizations and applications of the Callias-type index theorem were obtained recently in \cite{Kottke11,CarvalhoNistor14,Wimmer14,Kottke15,BrShi16,BrCecchini17}. 

P.~Shi, \cite{Shi17}, proved a version of the Callias-type index theorem for the  APS  boundary value problem for Callias-type operators on a complete odd-dimensional manifold with compact boundary. 

Fox and Haskell \cite{FoxHaskell03,FoxHaskell05} studied Callias-type operators on manifolds with non-compact boundary.  Under rather strong conditions on the geometry of the manifold and the operator $\DD$ they proved a version of the Atiyah--Patodi--Singer index theorem. 

In \cite{BrShi17} we studied  the index of the APS boundary value problem on an arbitrary complete odd-dimensional manifold with non-compact boundary. In the current paper, we obtain an even-dimensional analogue of \cite{BrShi17}.

\subsection{An almost compact essential support}\label{SS:Iesssupport}

A manifold $C$, whose boundary is a disjoint union of two complete manifolds $N_0$ and $N_1$, is called {\em essentially cylindrical} if outside of a compact set it is isometric to a cylinder $[0,\varepsilon]\times N$, where $N$ is a non-compact manifold. It follows that manifolds $N_0$ and $N_1$ are isometric outside of a compact set. 

We say that an essentially cylindrical manifold $M_1$, which contains $\pM$,  is an {\em almost compact essential support of\/ $\DD$}\/ if the restriction of $\DD^*\DD$ to $M\setminus M_1$ is strictly positive and the restriction of $\DD$ to the cylinder $[0,\varepsilon]\times N$ is a product, cf. Definition~\ref{D:almost compact support}. Every strongly Callias-type operator on $M$ which is a product near $\pM$ has an almost compact essential support. 

Theorem~\ref{T:indM=indM1} states that {\em the index of the APS boundary value problem for a strongly Callias-type operator $\DD$ on a complete manifold $M$ is equal to the index of the APS boundary value problem of the restriction of $\DD$ to its almost compact essential support $M_1$}.

\subsection{Index on an essentially cylindrical manifold}\label{SS:Iesscyl}

Let $M$ be an essentially cylindrical manifold and let $\DD$ be a strongly Callias-type operator on $M$, whose restriction to the cylinder $[0,\varepsilon]\times \p M$ is a product. Suppose $\pM= N_0\sqcup N_1$ and denote the restrictions of $\DD$ to $N_0$ and $N_1$ by $\AAA_0$ and $-\AAA_1$ respectively (the sign convention means that we think  of  $N_0$ as the ``left boundary'' and of $N_1$ as the ``right boundary'' of $M$). Let $\DD_B$ denote the operator $\DD$ with APS boundary conditions. 

Let $\AS(\DD)$ denote the Atiyah--Singer integrand of $\DD$. This is a differential form on $M$ which depends on the geometry of the manifold and the bundle. Since all structures are product outside of the compact set $K$, this form vanishes outside of $K$. Hence, $\int_M\AS(\DD)$ is well-defined and finite. Our main result here is that 
\begin{equation}\label{E:Iess compact}
	\ind \DD_B\ - \ \int_M\,\AS(\DD)
\end{equation}
{\em  depends only on the operators $\AAA_0$ and $\AAA_1$} and does  not depend  on the interior of the manifold $M$ and the restriction of\/ $\DD$ to the interior  of $M$, cf. Theorem~\ref{T:indep of D}.

\subsection{The relative $\eta$-invariant}\label{SS:Ireleta}
Suppose now that $\AAA_0$ and $\AAA_1$ are self-adjoint strongly Callias-type operators on complete manifolds $N_0$ and $N_1$. An {\em almost compact cobordism} between $\AAA_0$ and $\AAA_1$ is a pair $(M,\DD)$ where $M$ is an essentially cylindrical manifold with $\pM=N_0\sqcup N_1$ and $\DD$ is a strongly Callias-type operator on $M$, whose restriction to the cylindrical part of $M$ is a product and such that the restrictions of $\DD$ to $N_0$ and $N_1$ are equal to $\AAA_0$ and $-\AAA_1$ respectively. We say that $\AAA_0$ and $\AAA_1$ are {\em cobordant}\/ if there exists an almost compact cobordism between them. In particular, this implies that  $\AAA_0$ and $\AAA_1$ are equal outside of a compact set. 

Let  $\DD$ be an almost compact cobordism between $\AAA_0$ and $\AAA_1$. Let $B_0$ and $B_1$ be the APS boundary conditions for $\DD$ at $N_0$ and $N_1$ respectively. Let $\ind \DD_{B_0\oplus B_1}$ denote the index of the APS boundary value problem for $\DD$.  We define the {\em relative $\eta$-invariant} by the formula
\begin{equation}\label{E:Ireleta}\notag
	\eta(\AAA_1,\AAA_0) \ = \  
	2\,\left(\,\ind \DD_{B_0\oplus B_1}-\int_M\,\AS(\DD)\,\right)
	 \ + \ \dim\ker \AAA_0\ + \ \dim\ker \AAA_1.
\end{equation}
It follows from the result of the previous subsection that $\eta(\AAA_1,\AAA_0)$ is independent of the choice of an almost compact cobordism.  

If $M$ is a compact manifold, then the Atiayh-Patodi-Singer index theorem \cite{APS1} implies that $\eta(\AAA_1,\AAA_0)= \eta(\AAA_1)-\eta(\AAA_0)$. In general, for non-compact manifolds, the individual $\eta$-invariants $\eta(\AAA_1)$ and $\eta(\AAA_0)$ might not be defined. However, $\eta(\AAA_1,\AAA_0)$  behaves like it were a difference of two individual $\eta$-invariants. In particular, cf. Propositions~\ref{P:antisymmetry eta}--\ref{P:cocycle},  
\[
	\eta(\AAA_1,\AAA_0)\ = \ -\,\eta(\AAA_0,\AAA_1), 
	\qquad
	\eta(\AAA_2,\AAA_0)\ = \ \eta(\AAA_2,\AAA_1)\ + \ \eta(\AAA_1,\AAA_0).
\]

Under rather strong conditions on the manifolds $N_0$ and $N_1$ and on the operators $\AAA_0$, $\AAA_1$,  Fox and Haskell \cite{FoxHaskell03,FoxHaskell05} showed that the  heat kernel of $\AAA_j$ ($j=0,1$) has a nice asymptotic expansion similar to the one  for operators on compact manifolds. Then they were able to define the  individual $\eta$-invariants $\eta(\AAA_j)$ ($j=0,1$). In this situation, as expected, our relative $\eta$-invariant is equal to the difference of the individual $\eta$-invariants:  $\eta(\AAA_1,\AAA_0)= \eta(\AAA_1)-\eta(\AAA_0)$. 

Under much weaker (but still quite strong) assumptions,  M\"uller, \cite{Muller98}, suggested a definition of a relative $\eta$-invariant based on analysis of the relative heat kernel \cite{Donnelly87,Bunke95,Bunke93}. This invariant behaves very similar to our $\eta(\AAA_1,\AAA_0)$. The precise conditions under which these two invariants are equal are not clear yet. We note that our invariant is defined on a much wider class of manifolds, where the relative heat kernel is not of trace class and can not be used to construct an $\eta$-invariant.

\subsection{The spectral flow}\label{SS:Isp flow}

Consider a family $\AA= \{\AAA^s\}_{0\le s\le 1}$ of self-adjoint strongly Callias-type operators  on a complete Riemannian manifold. We assume that there is a compact set $K\subset M$ such that the restriction of $\AAA^s$ to $M\setminus K$ is independent of $s$. Since the spectrum of $\AAA^s$ is discrete for all $s$, the spectral flow $\spf(\AA)$ can be defined in a more or less usual way. By Theorem~\ref{T:sp flow}, if $\AAA_0$ is a self-adjoint strongly Callias-type operator which is cobordant to $\AAA^0$ (and hence, to all $\AAA^s$), then
\begin{equation}\label{E:Isp flow}
		\eta(\AAA^1,\AAA_0)\ - \ \eta(\AAA^0,\AAA_0)
	\ - \ \int_0^1\,\big(\frac{d}{ds}\oeta(\AAA^s,\AAA_0)\big)\,ds
	\ = \ 2\,\spf(\AA).
\end{equation}
The derivative $\frac{d}{ds}\oeta(\AAA^s,\AAA_0)$ can be computed as an integral of the transgression form --- a differential form canonically constructed from the symbol of $\AAA^s$ and its derivative with respect to $s$. Thus \eqref{E:Isp flow} expresses the change of the relative $\eta$-invariant as a sum of $2\spf(\AA)$ and a local differential geometric expression.

\section{Boundary value problems for Callias-type operators}\label{S:setting}
In this section we recall some results about boundary value problems for Callias-type  operators on manifolds with non-compact boundary,  \cite{BrShi17}, keeping in mind applications to even-dimensional case. The operators considered here are slightly more general than those discussed in \cite{BrShi17}, but all the definitions and most of the properties of the boundary value problems remain the same.

\subsection{Self-adjoint strongly Callias-type operators}\label{SS:saCallias}

Let $M$ be a complete Riemannian manifold (possibly with boundary) and let $E\to M$ be a Dirac bundle over $M$, cf. \cite[Definition~II.5.2]{LawMic89}. In particular, $E$ is a Hermitian bundle endowed with a Clifford multiplication $c:T^*M\to \End(E)$ and a compatible Hermitian connection $\n^E$.  Let $D:C^\infty(M,E)\to C^\infty(M,E)$ be the Dirac operator defined by the connection $\n^E$. Let $\Psi\in{\rm End}(E)$ be a self-adjoint bundle map (called a \emph{Callias potential}). Then
\[
	\DD\;:=\;D+\Psi
\]
is a formally self-adjoint Dirac-type operator on $E$ and
\begin{equation}\label{E:Calliaseq}
	\DD^2\;=\;D^2+\Psi^2+[D,\Psi]_+,
\end{equation}
where $[D,\Psi]_+:=D\circ\Psi+\Psi\circ D$ is the anticommutator of the operators $D$ and $\Psi$.

\begin{definition}\label{D:saCallias}
We call $\DD$ a \emph{self-adjoint strongly Callias-type operator} if
\begin{enumerate}
\item $[D,\Psi]_+$ is a zeroth order differential operator, i.e. a bundle map;
\item for any $R>0$, there exists a compact subset $K_R\subset M$ such that
\begin{equation}\label{E:strinvinfA}
	\Psi^2(x)\;-\;\big|[D,\Psi]_+(x)\big|\;\ge\;R
\end{equation}
for all $x\in M\setminus K_R$.  In this case, the compact set $K_R$ is called an \emph{$R$-essential support} of $\DD$, or an \emph{essential support} when we do not need to stress the associated constant.
\end{enumerate}
\end{definition}

\begin{remark}\label{R:Psi anticommutes}
Condition (i) of Definition~\ref{D:saCallias} is equivalent to the condition that $\Psi$ anticommutes with the Clifford multiplication: $\big[c(\xi),\Psi\big]_+ = 0$, for all $\xi\in T^*M$.
\end{remark}

\subsection{Graded self-adjoint strongly Callias-type operators}\label{SS:grCallias}

Suppose now that $E=E^+\oplus E^-$ is a $\ZZ_2$-graded Dirac bundle such that the Clifford multiplication $c(\xi)$ is odd and the Clifford connection is even with respect to this grading. Then
\[
D\;:=\;\left(
\begin{matrix}
0 & D^- \\
D^+ & 0 
\end{matrix}
\right)
\]
is the $\ZZ_2$-graded Dirac operator, where
\(
D^\pm\;:\;C^\infty(M,E^\pm)\to C^\infty(M,E^\mp)
\)
are formally adjoint to each other. Assume that the Callias potential $\Psi$ has odd grading degree, i.e., 
\[
\Psi\;=\;\left(
\begin{matrix}
0 & \Psi^- \\
\Psi^+ & 0 
\end{matrix}
\right),
\]
where $\Psi^\pm\in\Hom(E^\pm,E^\mp)$ are adjoint to each other. Then we have
\begin{equation}\label{E:Callias}
	\DD\;=\;D\,+\,\Psi\;=\;\left(
	\begin{matrix}
	0 & D^-+\Psi^- \\
	D^++\Psi^+ & 0
	\end{matrix}
	\right)\;=:\;\left(
	\begin{matrix}
	0 & \DD^- \\
	\DD^+ & 0
	\end{matrix}
	\right).
\end{equation}

\begin{definition}\label{D:grCallias}
Under the same condition as in Definition \ref{D:saCallias}, $\DD$ is called a \emph{graded self-adjoint strongly Callias-type operator}. In this case, we also call $\DD^+$ and $\DD^-$ strongly Callias-type operators. They are formally adjoint to each other. By an \emph{$R$-essential support} (or \emph{essential support}) of $\DD^\pm$ we understand as an $R$-essential support (or essential support) of $\DD$. 
\end{definition}

\begin{remark}\label{R:grCallias-even}
When $M$ is an oriented even-dimensional manifold there is a natural grading of $E$ induced by the volume form. We will consider this situation in the next section.
\end{remark}

\begin{remark}\label{R:grCallias-odd}
Suppose there is a skew-adjoint isomorphism $\gamma:E^\pm\to E^\mp$, $\gamma^*=-\gamma$, which anticommutes with multiplication $c(\xi)$ for all $\xi\in T^*M$, satisfies $\gamma^2=-1$, and is flat with respect to the connection $\n^E$, i.e. $[\n^E,\gamma]=0$. Then $\xi\mapsto \gamma\circ c(\xi)$ defines a Clifford multiplication of $T^*M$ on $E^+$ and the corresponding Dirac operator is $\tilde{D}^+= \gamma\circ D^+$. Suppose also that $\gamma$ commutes with $\Psi$. Then $\Phi^+=-i\gamma\circ\Psi^+$ is a self-adjoint endomorphism of $E^+$. In this situation, 
\[
	\tilde{D}^++i\Phi^+ =  \gamma\circ\DD^+:\, 
	C^\infty(M,E^+)\ \to \ C^\infty(M,E^+)
\] 
is a strongly Callias-type operator in the sense of \cite[Definition 3.4]{BrShi17}.
\end{remark}

\subsection{Restriction to the boundary}\label{SS:restriction}

Assume that the Riemannian metric $g^M$ is {\em product near the boundary}, that is, there exists a neighborhood $U\subset M$ of the boundary which is isometric to the cylinder
\begin{equation}\label{E:Zr}
	Z_r\ := \ [0,r)\times\dm. 
\end{equation} 
In the following we identify $U$ with $Z_r$ and denote by $t$ the coordinate along the axis of $Z_r$. Then the inward unit normal one-form to the boundary is given by $\tau = dt$. 

Furthermore, we assume that the Dirac bundle $E$ is {\em product near the boundary}. In other words we assume that the Clifford multiplication $c:T^*M\to{\rm End}(E)$ and the connection $\nabla^E$ have product structure on $Z_r$, cf. \cite[\S3.7]{BrShi17}.

Let $D$ be a $\ZZ_2$-graded Dirac operator. In this situation the restriction of $D$ to $Z_r$ takes the form 
\begin{equation}\label{E:productD}
	D\; = \;c(\tau)(\p_t+\hat{A})=
	\left(
	\begin{matrix}
	0 & c(\tau) \\
	c(\tau) & 0
	\end{matrix}
	\right)
	\left(
	\begin{matrix}
	\p_t+A & 0 \\
	0 & \p_t+A^\sharp
	\end{matrix}
	\right), 
\end{equation}
where
\[
	A:\;C^\infty(\dm,E^+_{\dm})\;\to\; C^\infty(\dm,E^+_{\dm})
\]
and 
\begin{equation}\label{E:anticommuting}
	A^\sharp= c(\tau)\circ A\circ c(\tau):\, 
	C^\infty(\dm,E^-_{\dm})\;\to\; C^\infty(\dm,E^-_{\dm})
\end{equation}
are formally self-adjoint operators acting on the restrictions of $E^\pm$ to the boundary.

\begin{remark}\label{R:Asharp}
It would be more natural to use the notation $A^+$ and $A^-$ instead of $A$ and $A^\sharp$. But since in the future we only deal with the operator $A:C^\infty(\dm,E^+_{\dm})\to C^\infty(\dm,E^+_{\dm})$ we remove the superscript ``$+$'' to simplify the notation. 
\end{remark}

Let $\DD=D+\Psi$ be a graded self-adjoint strongly Callias-type operator. Then the restriction of $\DD$ to $Z_r$ is given by 
\begin{equation}\label{E:productDD}
	\DD\; = \;c(\tau)(\p_t+\hat\AAA)\;=\;
	\left(
	\begin{matrix}
	0 & c(\tau) \\
	c(\tau) & 0
	\end{matrix}
	\right)
	\left(
	\begin{matrix}
	\p_t+\AAA & 0 \\
	0 & \p_t+\AAA^\sharp
	\end{matrix}
	\right), 
\end{equation}
where
\begin{equation}\label{E:AAA}
		\AAA \;:= \;A-c(\tau)\Psi^+:\,C^\infty(\dm,E^+_{\dm})
		\;\to \;C^\infty(\dm,E^+_{\dm}).
\end{equation}
and $\AAA^\sharp = A^\sharp-c(\tau)\Psi^-$. 
By Remark \ref{R:Psi anticommutes}, $c(\tau)\Psi^\pm\in{\rm End}(E^\pm_{\dm})$ are self-adjoint bundle maps. Therefore $\AAA$ and $\AAA^\sharp$ are formally self-adjoint operators. In fact, they are strongly Callias-type operators, cf. Lemma~3.12 of \cite{BrShi17}.  In particular, they have discrete spectrum. Also,
\[
	\AAA^\sharp \ = \  c(\tau)\circ\AAA\circ c(\tau).
\]

\begin{definition}\label{D:producDD}
We say that a graded self-adjoint strongly Callias-type operator $\DD$ is {\em product near the boundary} if the Dirac bundle $E$ is  product near the boundary and the restriction of the Callias potential $\Psi$ to $Z_r$ does not depend on $t$. The operator $\AAA$ (resp. $\AAA^\sharp$)  of \eqref{E:productDD} is called the {\em restriction of $\DD^+$ (resp. $\DD^-$) to the boundary}. 
\end{definition}

\subsection{Sobolev spaces on the boundary}\label{SS:Sob&specproj}

Consider a graded  self-adjoint  strongly Callias-type operator $\DD:C^\infty(M,E)\to C^\infty(M,E)$, cf. \eqref{E:Callias}. The restriction of $\DD^+$ to the boundary is a self-adjoint strongly Callias-type operator
\[
	\AAA:\,C^\infty(\dm,E_{\dm}^+)\ \to \ C^\infty(\dm,E_{\dm}^+).
\] 
We recall the definition of Sobolev spaces $H^s_\AAA(\pM,E^+_\pM)$ of sections over $\pM$ which depend on the boundary operator $\AAA$, cf. \cite[\S3.13]{BrShi17}.

\begin{definition}\label{D:Sobolev}
Set
\[
	C_\AAA^\infty(\dm,E_{\dm}^+)	\;:=\;
	\Big\{\,\bfu\in C^\infty(\dm,E_{\dm}^+):\,
	\big\|(\id+\AAA^2)^{s/2}\bfu\big\|_{L^2(\dm,E_\pM)}^2<+\infty\mbox{ for all }s\in\RR\,\Big\}.
\]
For all $s\in\RR$ we define the \emph{Sobolev $H_\AAA^s$-norm} on $C_\AAA^\infty(\dm,E_{\dm}^+)$ by
\begin{equation}\label{E:Sobnorm}
	\|\bfu\|_{H_\AAA^s(\dm,E_\pM)}^2\;:=\;
		\big\|(\id+\AAA^2)^{s/2}\bfu\big\|_{L^2(\dm,E_\pM)}^2.
\end{equation}
The Sobolev space $H_\AAA^s(\dm,E_{\dm}^+)$ is defined to be the completion of $C_\AAA^\infty(\dm,E_{\dm}^+)$ with respect to this norm.
\end{definition}

\subsection{Generalized APS boundary conditions}\label{SS:gAPS}

The eigensections of $\AAA$ belong to $H_\AAA^s(\pM,E_\pM^+)$ for all $s\in \RR$, cf. \cite[\S3.17]{BrShi17}. For $I\subset \RR$ we denote by 
\[
	H_I^s(\AAA)\ \subset \  H_\AAA^s(\dm,E_{\dm}^+)
\] 
the span of the eigensections of $\AAA$ whose eigenvalues belong to $I$.

\begin{definition}\label{D:gAPS}
For any $a\in\RR$, the subspace
\begin{equation}\label{E:gAPS}
B\;=\;B(a)\;:=\;H_{(-\infty,a)}^{1/2}(\AAA).
\end{equation}
is called the \emph{the generalized Atiyah--Patodi--Singer boundary conditions} for $\DD^+$. If $a=0$, then the space $B(0)= H_{(-\infty,0)}^{1/2}(\AAA)$ is called the \emph{Atiyah--Patodi--Singer (APS) boundary condition}. 

The spaces $\oB(a) := H_{(-\infty,a]}^{1/2}(\AAA)$  and  $\oB(0) := H_{(-\infty,0]}^{1/2}(\AAA)$ are called the \emph{dual generalized APS boundary conditions} and the {\em dual APS boundary conditions} respectively.

The space $B^\ad=B^\ad(a):=  H_{(-\infty,-a]}^{1/2}(\AAA^\sharp)$ is called the {\em adjoint of the generalized APS boundary condition for $\DD^+$}. One can see that it is a dual generalized APS boundary condition for $\DD^-$.
\end{definition}

\begin{definition}\label{D:gAPSbvp}
If $B$ is a generalized APS boundary condition for $\DD^+$, we denote by $\DD^+_B$  the operator $\DD^+$ with domain
\[
	\dom\DD^+_B\;:=\;\{u\in \dom \DD^+_{\max}:\,u|_{\dm}\in B\},
\]
where $\dom \DD^+_{\max}$ denotes the domain of the maximal extension of $\DD^+$ (cf. \cite{BrShi17}). 
We refer to $\DD^+_B$  as the \emph{generalized APS boundary value problem} for $\DD^+$. 
\end{definition}

Recall that $\DD^-$ is the formal adjoint of $\DD^+$. It is shown in Example~4.9 of \cite{BrShi17} that the  $L^2$-adjoint of $\DD^+_{B(a)}$  is given by $\DD^-$ with the dual APS boundary condition $B^\ad(a)$:
\begin{equation}\label{E:dual}
	\big(\DD^+_{B(a)}\big)^\ad  
	\ = \ \DD^-_{B^\ad(a)}.
\end{equation}

\begin{theorem}\label{T:fredholm}
Suppose that a graded strongly Callias-type operator \eqref{E:Callias} is product near $\pM$. Then the operator $\DD^+_B:\,\dom\DD^+_B\to L^2(M,E^-)$ is Fredholm. In particular, it has finite dimensional kernel and cokernel.
\end{theorem}
\begin{proof}
For the case discussed in Remark~\ref{R:grCallias-odd} this is proven in Theorem~5.4 of \cite{BrShi17}. Exactly the same proof works in the general case.
\end{proof}

\subsection{The index of generalized APS boundary value problems}\label{SS:indexAPS}

By \eqref{E:dual} the cokernel of $\DD^+_{B(a)}$ is isomorphic to the kernel of $\DD^-_{B^\ad(a)}$.

\begin{definition}\label{D:index}
Let $\DD^+$ be a strongly Callias-type operator on a complete Riemannian manifold $M$ which is product near the boundary. Let $B= H_{(-\infty,a)}^{1/2}(\AAA)$ be a generalized APS boundary condition for $\DD^+$ and let $B^\ad= H_{(-\infty,-a]}^{1/2}(\AAA^\sharp)$ be the adjoint of the generalized APS boundary condition. The integer 
\begin{equation}\label{E:def of index}
		\ind\DD^+_B\;:=\;\dim\ker\DD^+_B-\dim\ker(\DD^-)_{B^\ad}\;\in\;\ZZ
\end{equation}
is called the {\em index of the boundary value problem} $\DD^+_B$. 
\end{definition}

It follows directly from \eqref{E:dual} that 
\begin{equation}\label{E:indD-indD*}
	\ind\, (\DD^-)_{B^\ad}\ = \ -\ind \DD^+_B.
\end{equation}

\subsection{More general boundary value conditions}\label{SS:gen bc}
Generalized APS and dual generalized APS boundary conditions are examples of {\em elliptic boundary conditions}, \cite[Definition~4.7]{BrShi17}. In this paper we don't work with general elliptic boundary conditions. However, in Section~\ref{S:sp flow} we need a slight modification of APS boundary conditions, which we define now. 

\begin{definition}\label{D:rel index}
We say that two closed subspaces $X_1$, $X_2$ of a Hilbert space $H$ are {\em finite rank perturbations}  of each other if there exists a finite dimensional subspace $Y\subset H$ such that $X_2\subset X_1\oplus Y$ and the quotient space $(X_1\oplus Y)/X_2$ has finite dimension. 

The {\em relative index} of $X_1$ and $X_2$ is defined by 
\begin{equation}\label{E:rel index}
	[X_1,X_2] \ := \ \dim\, (X_1\oplus Y)/X_2\ - \ \dim Y.
\end{equation}
\end{definition}

One easily sees that the relative index is independent of the choice of $Y$. We  shall need the following analogue of \cite[Proposition~5.8]{BrShi17}:

\begin{proposition}\label{P:change of B}
Let $\DD$ be a graded self-adjoint strongly Callias-type operator on $M$ and let $B$ be a generalized APS or dual generalized APS boundary condition for $\DD^+$.  If $B_1\subset H^{1/2}_\AAA(\pM,E^+_\pM)$ is a finite rank perturbation of $B$, then the operator $\DD^+_{B_1}$ is Fredholm and  
\begin{equation}\label{E:change of B}
	\ind \DD^+_{B}\ - \ \ind \DD^+_{B_1}\ = \ [B,B_1].
\end{equation}
\end{proposition}
The proof of the proposition is a verbatim repetition of the proof of \cite[Theorem~8.14]{BaerBallmann12}.

As an immediate consequence of Proposition~\ref{P:change of B} we obtain the following
\begin{corollary}\label{C:change of B}
Let $\AAA$ be the restriction of $\DD^+$ to $\pM$ and let  $B_0= H^{1/2}_{(-\infty,0)}(\AAA)$ and $\oB_0= H^{1/2}_{(-\infty,0]}(\AAA)$ be the APS and the dual APS boundary conditions respectively. Then 
\begin{equation}\label{E:change of B2}
	\ind \DD^+_{\oB_0} \ = \ \ind\DD^+_{B_0}\ + \ \dim\ker \AAA.
\end{equation}
\end{corollary}

\subsection{The splitting theorem}\label{SS:splittingthm}

Let $M$ be a complete manifold. Let $N\subset M$ be a hypersurface disjoint from $\dm$ such that cutting $M$ along $N$ we obtain a manifold $M'$ (connected or not) with $\dm$ and two copies of $N$ as boundary. So we can write $M'=(M\setminus N)\sqcup N_1\sqcup N_2$.

Let $E=E^+\oplus E^-\to M$ be a $\ZZ_2$-graded Dirac bundle over $M$ and $\DD^\pm:C^\infty(M,E^\pm)\to C^\infty(M,E^\mp)$ be strongly Callias-type operators as in Subsection \ref{SS:grCallias}. They induce $\ZZ_2$-graded Dirac bundle $E'=(E')^+\oplus(E')^-\to M'$ and strongly Callias-type operators 
\[
	(\DD')^\pm:C^\infty(M',(E')^\pm)\ \to \ C^\infty(M',(E')^\mp)
\]
on $M'$. We assume that all structures are product near $N_1$ and $N_2$. Let $\AAA$ be the restriction of $(\DD')^+$ to $N_1$. Then $-\AAA$ is the restriction of $(\DD')^+$ to $N_2$ and, thus, the restriction of $(\DD')^+$ to $N_1\sqcup N_2$ is $\AAA'=\AAA\oplus(-\AAA)$. The following {\em Splitting Theorem} is an analogue of Theorem 5.11 of \cite{BrShi17} with the same proof.

\begin{theorem}\label{T:splitting}
Suppose $M,\DD^+,M',(\DD')^+$ are as above. Let $B_0$ be a generalized APS boundary condition on $\dm$. Let $B_1=H_{(-\infty,0)}^{1/2}(\AAA)$ and $B_2= H_{(-\infty,0]}^{1/2}(-\AAA)= H_{[0,\infty)}^{1/2}(\AAA)$ be the APS and the dual APS boundary conditions for $(\DD')^+$ along $N_1$ and $N_2$, respectively. Then $(\DD')^+_{B_0\oplus B_1\oplus B_2}$ is a Fredholm operator and
\[
	\ind\DD^+_{B_0}\;=\;\ind(\DD')^+_{B_0\oplus B_1\oplus B_2}.
\]
\end{theorem}

\subsection{Reduction of the index to an essentially cylindrical manifold}\label{SS:esscylindrical}

The study of the index of Callias-type operators on manifolds without boundary can be reduced to a computation on the essential support. For manifolds with boundary we want an analogous subset, but the one which contains the boundary. Such a set is necessarily non-compact, but we want it to be ``similar to a compact set". In \cite[Section 6]{BrShi17}, we introduce the notion of essentially cylindrical manifolds, which replace the role of compact subsets in our study of boundary value problems.

\begin{definition}\label{D:esscylindrical}
An {\em essentially cylindrical} manifold  $C$ is a complete Riemannian manifold whose boundary is a union of two disjoint manifolds, $\partial C=N_0\sqcup N_1$, such that 
\begin{enumerate}
\item
there exist a compact set $K\subset C$, an open Riemannian manifold $N$, and an isometry $C\setminus K\simeq [0,\delta]\times N$;

\item
under the above isometry $N_0\setminus K =\{0\}\times N$ and $N_1\setminus K =\{\delta\}\times N$. 
\end{enumerate}
See Figure 1 for an example of essentially cylindrical manifolds.
\end{definition}

\begin{center}
\begin{overpic}[height=5cm]{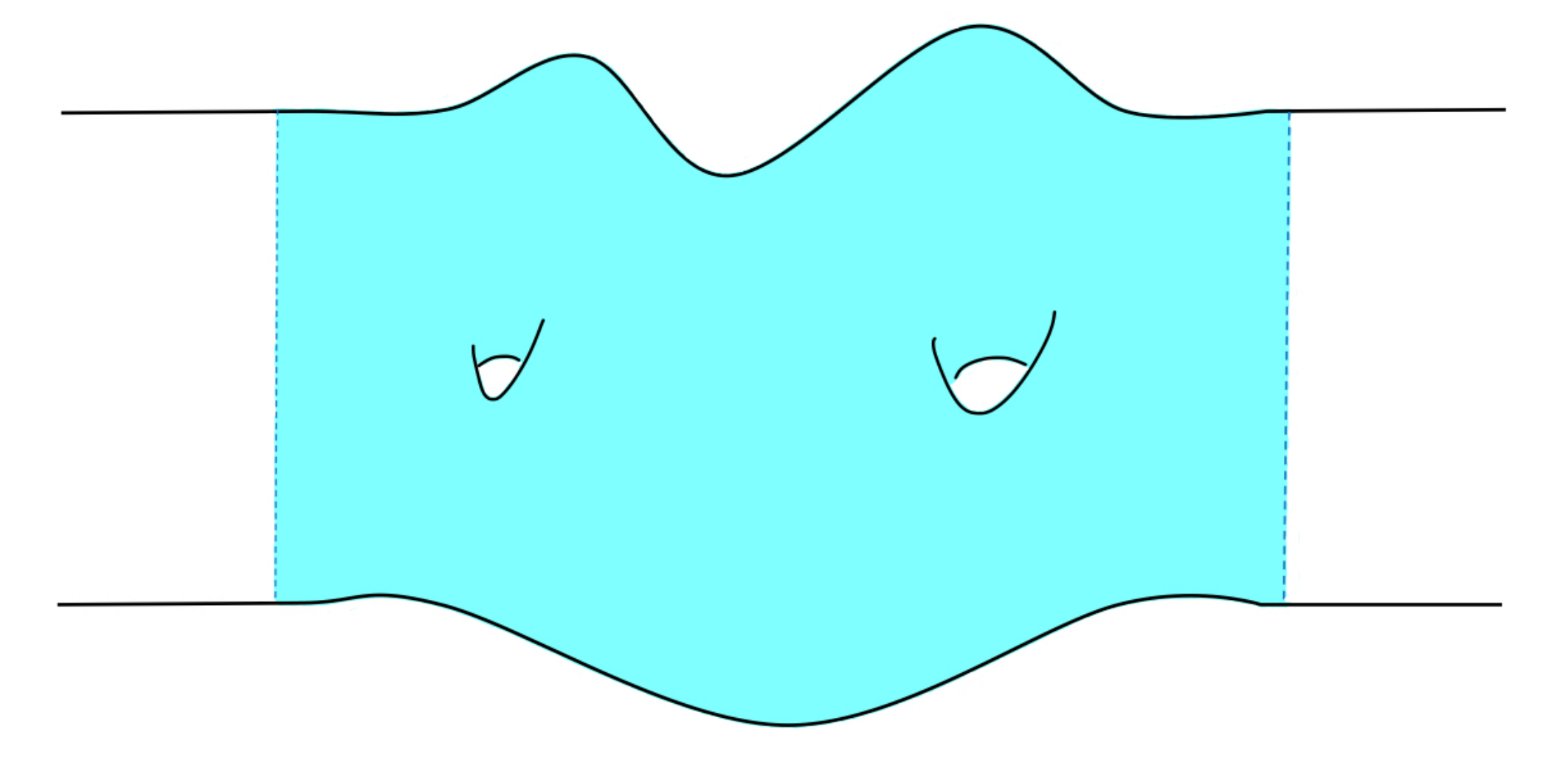}
\put(98,40){$N_0$}
\put(98,8){$N_1$}
\put(47,23){$K$}
\end{overpic}\\
Figure 1. An essentially cylindrical manifold  $C$
\end{center}

\begin{definition}\label{D:almost compact support}
Let $\DD^+$ be a strongly Callias-type operator on $M$. An {\em almost compact essential support of\, $\DD^+$} is a smooth submanifold $M_1\subset M$  with smooth boundary, which contains $\pM$ and such that there exist a (compact) essential support $K\subset M$ and $\varepsilon\in (0,r)$ such that 
\begin{equation}\label{E:almost compact}
	M_1\setminus K \ = \ (\pM\setminus K)\times [0,\varepsilon] 
	\subset Z_r. 
\end{equation}
\end{definition}

An almost compact essential support is a special type of essentially cylindrical manifolds. It is shown in \cite[Lemma~6.5]{BrShi17} that for any strongly Callias-type operator which is product near $\pM$ there exists an  almost compact essential support. Figure 2 illustrates the idea of building an almost compact essential support.

\begin{center}
\begin{overpic}[height=8cm]{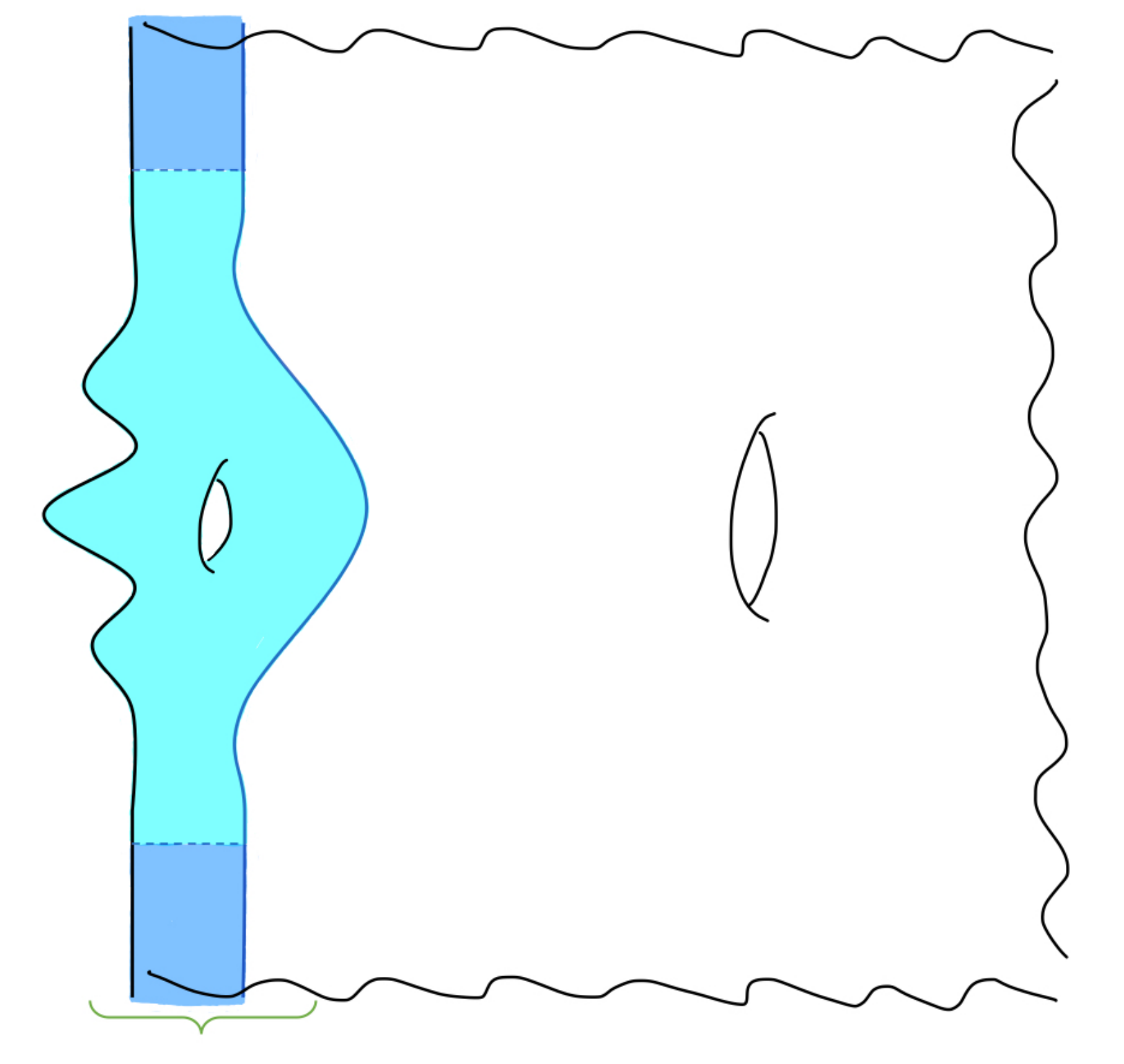}
\put(-5,47){$\pM$}
\put(34,47){$N_1$}
\put(15,36){$K$}
\put(15,-3){$M_1$}
\put(50,-4){$M$}
\end{overpic}\\
\vspace{.4cm}
Figure 2. The idea to obtain an almost compact essential support $M_1$
\end{center}

\subsection{The index on an almost compact essential support}\label{SS:index ac}
Suppose $M_1\subset M$ is an almost compact essential support of $\DD^+$. Set $N_0:=\pM$ and let $N_1\subset M$ be such that $\pM_1=N_0\sqcup N_1$ as in Definition \ref{D:esscylindrical}. The restriction of $\DD$ to a neighborhood of $N_1$ need not be product. Since in this paper we only consider boundary value problems for operators which are product near the boundary, we first deform $\DD$ to a product form. It is shown in \cite[Lemma~6.8]{BrShi17} that there exists a perturbation of all the structures such that 
\begin{enumerate}
\item 
the new structures are product near $N_1$. In particular, the corresponding Callias-type operator  $\DD'$ is product near $N_1$.
\item 
$\DD'$ has a compact essential support, which is contained in $M_1$.
\item
The difference $\DD-\DD'$ vanishes near $\pM$ and outside of a compact subset of $M_1$. In this situation we say that $\DD'$ (or $(\DD')^\pm$) is  a {\em compact perturbation of $\DD$ (or $\DD^\pm$)}. 
\end{enumerate}

Let $M_1\subset M$ be an almost compact essential support of $\DD^+$. Let $(\DD')^+$ be a compact perturbation of $\DD^+$ which is product near the boundary (cf. \cite[\S6.6]{BrShi17}). Let $\AAA$ be the restriction of $\DD^+$ to $\pM$. It is also the restriction of $(\DD')^+$. We denote by $-\AAA_1$ the restriction of $(\DD')^+$ to $N_1$. Theorem 6.10 of \cite{BrShi17} claims that the index of an elliptic boundary value problem on $M$ can be reduced to an index on an almost compact essential support:

\begin{theorem}\label{T:indM=indM1}
Suppose $M_1\subset M$ is an almost compact essential support of\/ $\DD^+$ and let\/ $\pM_1=\pM\sqcup N_1$.  Let\/ $(\DD')^+$ be a compact perturbation of\/ $\DD^+$ which is product near $N_1$ and  such that there is a compact essential support for $(\DD')^+$ which is contained in $M_1$. Let $B_0$ be a generalized APS boundary condition for $\DD^+$. View $(\DD')^+$ as an operator on $M_1$ and let 
\[
	B_1 \ = \ H_{(-\infty,0)}^{1/2}(-\AAA_1) \ = \ 
	H_{(0,\infty)}^{1/2}(\AAA_1)
\] 
be the APS boundary condition for $(\DD')^+$ at $N_1$. Then
\begin{equation}\label{E:indM=indM1}
	\ind \DD^+_{B_0}\ = \ \ind (\DD')^+_{B_0\oplus B_1}.
\end{equation}
\end{theorem}

\section{The index of operators on essentially cylindrical manifolds}\label{S:esscyl}

In this section we discuss the index of strongly Callias-type operators on even-dimensional essentially cylindrical manifolds. It is parallel to \cite[Section 7]{BrShi17}, where the odd-dimensional case was considered. 

From now on we assume that $M$ is an \emph{oriented even-dimensional} essentially cylindrical manifold whose boundary $\pM= N_0\sqcup N_1$ is a disjoint union of two non-compact manifolds $N_0$ and $N_1$. Let $E$ be a Dirac bundle over $M$. As pointed out in Remark \ref{R:grCallias-even}, there is a natural $\ZZ_2$-grading $E=E^+\oplus E^-$ on $E$. Let $\DD^+: C^\infty(M,E^+)\to C^\infty(M,E^-)$ be a strongly Callias-type operator as in Definition \ref{D:grCallias} (these data might or might not come as a restriction of another operator to its almost compact essential support. In particular, we don't assume that the restriction of $\DD^+$ to $N_1$ is invertible). Let $\AAA_0$ and $-\AAA_1$ be the restrictions of $\DD^+$ to $N_0$ and $N_1$ respectively. 

We first recall some definitions from \cite[Section 7]{BrShi17}.

\subsection{Compatible data}\label{SS:compatible}

Let $M$ be an essentially cylindrical manifold and let $\pM=N_0\sqcup N_1$. As usual, we identify a tubular neighborhood of $\pM$ with the product 
\[
	Z_r\ := \  \big(\, N_0\times[0,r)\,\big)\,\sqcup \, \big(\, N_1\times[0,r)\,\big)
	\ \subset \ M.
\]

\begin{definition}\label{D:essequal}
We say that another essentially cylindrical manifold $M'$ is {\em compatible} with  $M$ if there is a fixed isometry between $Z_r$ and a neighborhood $Z_r'\subset M'$ of the boundary of $M'$. 
\end{definition}

Note that if $M$ and $M'$ are compatible then their boundaries are isometric.

Let $M$ and $M'$ be compatible essentially cylindrical manifolds and let $Z_r$ and $Z_r'$ be as above. Let $E\to M$ be a $\ZZ_2$-graded Dirac bundle over $M$ and let  $\DD^+: C^\infty(M,E^+)\to C^\infty(M,E^-)$  be a strongly Callias-type operator whose restriction to $Z_r$ is product and such that $M$ is an almost compact essential support of $\DD^+$.  This  means that there is a compact set $K\subset M$ such that $M\setminus K= [0,\varepsilon]\times N$ and the restriction of $\DD^+$ to $M\setminus K$ is product (i.e.  is given by \eqref{E:productDD}). Let 
 $E'\to M'$ be a $\ZZ_2$-graded Dirac bundle over $M'$ and let $(\DD')^+: C^\infty(M',(E')^{+})\to C^\infty(M',(E')^{-})$ be a strongly Callias-type operator, whose restriction to  $Z_r'$ is product and such that  $M'$ is an almost compact essential support of $(\DD')^{+}$.

\begin{definition}\label{D:essequal operators}
In the situation discussed above we say that $\DD^+$ and $(\DD')^{+}$ are {\em compatible} if there is an isomorphism $E|_{Z_r}\simeq E'|_{Z_r'}$ of graded Dirac bundles which identifies the restriction of $\DD^+$ to $Z_r$ with the restriction of $(\DD')^{+}$ to $Z_r'$.  
\end{definition}

Let $\AAA_0$ and $-\AAA_1$ be the restrictions of $\DD^+$ to $N_0$ and $N_1$ respectively (the sign convention means that we think  of  $N_0$ as the ``left boundary'' and of $N_1$ as the ``right boundary'' of $M$). 
Let $B_0=H^{1/2}_{(-\infty,0)}(\AAA_0)$ and $B_1=H^{1/2}_{(-\infty,0)}(-\AAA_1)= H^{1/2}_{(0,\infty)}(\AAA_1)$ be the APS boundary conditions for $\DD^+$ at $N_0$ and $N_1$ respectively. Since $\DD^+$ and $(\DD')^{+}$ are equal near the boundary, $B_0$ and $B_1$ are also APS boundary conditions for $(\DD')^{+}$.

We denote by $\alpha_{\rm AS}(\DD^+)$ the Atiyah--Singer integrand of $\DD^+$. It can be written as
\[
	\alpha_{\rm AS}(\DD^+)\ := \ (2\pi i)^{-\dim M}\,\hat{A}(M)\,{\rm ch}(E/S)
\]
where $\hat{A}(M)$  and ${\rm ch}(E/S)$ are  the differential forms representing the $\hat{A}$-genus of $M$ and the {\em relative Chern character} of $E$, cf.  \cite[\S4.1]{BeGeVe}. Note that $\alpha_{\rm AS}(\DD^+)$ depends only on the metric on $M$ and the Clifford multiplication on $E$ and thus is independent of the potential $\Psi$.

Since outside of a compact set $K$, $M$ and $E$ are product, the interior multiplication by $\p/\p t$ annihilates  $\AS$. Hence, the top degree component of $\AS$ vanishes on $M\setminus K$. We conclude that the integral  $\int_M\alpha_{\rm AS}(\DD^+)$ is well-defined and finite. Similarly, $\int_{M'}\alpha_{\rm AS}((\DD')^+)$ is well-defined and finite.

\begin{theorem}\label{T:indep of D}
Suppose $\DD^+$ is a strongly Callias-type operator on an oriented even-dimensional essentially cylindrical manifold $M$ such that $M$ is an almost compact essential support of\, $\DD^+$. Suppose that the operator  $(\DD')^{+}$ is compatible with $\DD^+$. Let $\pM=N_0\sqcup N_1$ and let $B_0= H^{1/2}_{(-\infty,0)}(\AAA_0)$ and $B_1= H^{1/2}_{(-\infty,0)}(-\AAA_1)= H^{1/2}_{(0,\infty)}(\AAA_1)$ be the APS boundary conditions for $\DD^+$ (and, hence, for $(\DD')^{+}$) at  $N_0$ and $N_1$ respectively. Then 
\begin{equation}\label{E:indep of D}
	\ind \DD^+_{B_0\oplus B_1}\;-\;\int_M\alpha_{\rm AS}(\DD^+) 
	\ = \ 
	\ind (\DD')^+_{B_0\oplus B_1}\;-\;\int_{M'}\alpha_{\rm AS}((\DD')^{+}).
\end{equation}
In particular, $\ind \DD^+_{B_0\oplus B_1}-\int_M\alpha_{\rm AS}(\DD^+)$ depends only on the restrictions $\AAA_0$ and $-\AAA_1$ of $\DD^+$ to the boundary.
\end{theorem}

The rest of the section is devoted  to the  proof of this theorem.

\begin{remark}\label{R:compare with odd}
In \cite{BrShi17} the odd dimensional version of Theorem~\ref{T:indep of D} was considered. Of course, in this case $\AS$ vanishes identically and Theorem~7.5 of \cite{BrShi17} states that the indexes of compatible operators are equal. The proof in \cite{BrShi17} is based on an application of a Callias-type index theorem and can not be adjusted to our current situation. Consequently, a completely different proof is proposed below.
\end{remark}
\subsection{Gluing the data together}\label{SS:glue}

We follow \cite[\S7.6, \S7.7]{BrShi17} to glue  $M$ with $M'$ and $\DD^+$ with $(\DD')^{+}$.

Let $-M'$ denote  the  manifold $M'$ with the opposite orientation. We identify a neighborhood of the boundary of $-M'$ with the product 
\[
	-Z_r'\ := \ 
	\big(\, N_0\times(-r,0]\,\big)\,\sqcup \, \big(\, N_1\times(-r,0]\,\big)
\]
and consider the union 
\[
	\tilde{M} \ := \ M\cup_{N_0\sqcup N_1} (-M').
\]
Then $Z_{(-r,r)}:= Z_r\cup (-Z_r')$ is a subset of $\tilde{M}$ identified with the product 
\[
	\big(\, N_0\times(-r,r)\,\big)\,\sqcup \, \big(\, N_1\times(-r,r)\,\big).
\]
We note that $\tilde{M}$ is a complete Riemannian manifold without boundary.

Let $E_\pM=E^+_\pM\oplus E^-_\pM$ denote the restriction of $E=E^+\oplus E^-$ to $\pM$. The product structure on $E|_{Z_r}$ gives a grading-respecting isomorphism $\psi:E|_{Z_r}\to [0,r)\times E_\pM$. Recall that we identified $Z_r$ with $Z_r'$ and fixed an isomorphism between the restrictions of $E$ to $Z_r$ and $E'$ to $Z_r'$. By a slight abuse of notation we use this isomorphism to view $\psi$ also as an isomorphism $E'|_{Z'_r}\to [0,r)\times E_\pM$.

Let $\tilde{E}\to \tilde{M}$ be the vector bundle over $\tilde{M}$ obtained by gluing $E$ and $E'$ using the isomorphism $c(\tau):E|_\pM\to E'|_{\pM'}$. This means that we fix isomorphisms 
\begin{equation}\label{E:tilE=}
	\phi:\,\tilde{E}|_M\to E, \quad \phi':\,\tilde{E}|_{M'}\to E', 
\end{equation}
so that 
\[
	\begin{aligned}
	\psi\circ\phi\circ \psi^{-1} \ &= \ 
		\id: [0,r)\times E_\pM \ \to \ [0,r)\times E_\pM, \\
	\psi\circ\phi'\circ \psi^{-1} \ &= \ 
		1\times c(\tau): [0,r)\times E_\pM \ \to \ [0,r)\times E_\pM.
	\end{aligned}
\]
Note that the grading of $E$ is preserved while the grading of $E'$ is reversed in this gluing process. Therefore $\tilde E=\tilde E^+\oplus\tilde E^-$ is a $\ZZ_2$-graded bundle.

We denote by $c':T^*M'\to \End(E')$ the Clifford multiplication on $E'$ and set $c''(\xi):= -c'(\xi)$. Then $\tilde{E}$ is a Dirac bundle over $\tilde{M}$ with the Clifford multiplication
\begin{equation}\label{E:tildec}
	\tilde{c}(\xi)\ : = \ 
  \begin{cases}
  	&c(\xi), \qquad \xi\in T^*M;\\
  	&c''(\xi)=-c'(\xi), \qquad \xi\in T^*M'.
  \end{cases}
\end{equation}
One readily checks that \eqref{E:tildec} defines a smooth odd-graded Clifford multiplication on $\tilde{E}$. Let $\tilde D:C^\infty(\tilde M,\tilde E)\to C^\infty(\tilde M,\tilde E)$ be the $\ZZ_2$-graded Dirac operator. Then the isomorphism $\phi$ of \eqref{E:tilE=} identifies the restriction of $\tilde{D}^\pm$ with $D^\pm$, the isomorphism $\phi'$ identifies the restriction of $\tilde{D}^\pm$ with $-(D')^\mp$, and isomorphism $\psi\circ\phi'\circ\psi^{-1}$ identifies the restriction of $\tilde{D}^\pm$ to $-Z_r'$ with
\begin{equation}\label{E:tilDM'}
 	\tilde{D}^\pm|_{Z_r'}\ = \ - c'(\tau)\circ (D')^\pm_{Z'_r}\circ c'(\tau)^{-1}.
 \end{equation} 

Let $(\Psi')^\pm$ denote the Callias potentials of $(\DD')^\pm$, so that $(\DD')^\pm=(D')^\pm+(\Psi')^\pm$. Consider the bundle maps $\tilde{\Psi}^\pm\in \Hom (\tilde{E}^\pm,\tilde E^\mp)$ whose restrictions to $M$ are equal to $\Psi^\pm$ and whose restrictions to $M'$ are equal to $-(\Psi')^\mp$. The two pieces fit well on $Z_{(-r,r)}$ by Remark \ref{R:Psi anticommutes}. To sum up the constructions presented in this subsection, we have

\begin{lemma}\label{L:tildeD}
The operators $\tilde{\DD}^\pm :=  \tilde{D}^\pm +\tilde{\Psi}^\pm$ are strongly Callias-type operators on $\tilde{M}$, formally adjoint to each other, whose restrictions to $M$ are equal to $\DD^\pm$ and whose restrictions to $M'$ are equal to $-(D')^\mp-(\Psi')^\mp = -(\DD')^\mp$.
\end{lemma}

The operator $\tilde{\DD}^+$ is a strongly Callias-type operator on a complete Riemannian manifold without boundary. Hence, \cite{Anghel90}, it is Fredholm. We again denote by $\alpha_{\rm AS}(\tilde\DD^+)$ the Atiyah--Singer integrand of $\tilde\DD^+$. It is explained in the paragraph before Theorem~\ref{T:indep of D} that the integral $\int_{\tilde M}\alpha_{\rm AS}(\tilde\DD^+)$ is well defined.

\begin{lemma}\label{L:indtildeD}
$\ind \tilde{\DD}^+=\int_{\tilde M}\alpha_{\rm AS}(\tilde\DD^+)$. 
\end{lemma}

\begin{proof}
Since $\tilde{M}$ is a union of two essentially cylindrical manifolds, there exists a compact essential support $\tilde{K}\subset \tilde{M}$ of $\tilde{\DD}$ such that $\tilde{M}\setminus\tilde{K}$ is of the form $S^1\times N$. We can choose  $\tilde{K}$ to be large enough so that the restriction of $\tilde{\DD}$ to a neighborhood $W$ of  $\tilde{M}\setminus\tilde{K}\simeq S^1\times N$ is a product of an operator on $N$ and an operator on $S^1$. Then the restriction of $\AS$ to this neighborhood vanishes. We can also assume that $\tilde{K}$ has a smooth boundary $\Sigma=S^1\times L$. 

Let $\hat{\DD}^+$ be a compact perturbation of $\tilde{\DD}^+$ (cf. Subsection \ref{SS:index ac}) in $W$ which is product both near $\Sigma$ and on $W$ and whose essential support is contained in $\tilde{K}$. Then 
\[
	\ind\tilde{\DD}^+ \ = \ \ind\hat{\DD}^+.
\]
We cut $\tilM$ along $\Sigma$ and apply the Splitting Theorem~\ref{T:splitting}\footnote{Since $\Sigma$ is compact we can also use the splitting theorem  for compact hypersurfaces,   \cite[Theorem 8.17]{BaerBallmann12}.} to get
\begin{equation}\label{E:tilD splits}
	\ind\tilde\DD^+
	\;=\;
	\ind\hat\DD^+_{\tilde K}\;+\;\ind\hat\DD^+_{\tilM\setminus\tilde K},
\end{equation}
where $\ind\hat\DD^+_{\tilde K}$ stands for the index of the restriction of $\hat\DD^+$ to $\tilde K$ with APS boundary condition, and $\ind\hat\DD^+_{\tilM\setminus\tilde K}$ stands for the index of the restriction of $\hat\DD^+$ to $\tilM\setminus\tilde K$ with the dual APS boundary condition.

Since $\hat\DD^+$ has an empty essential support in $\tilM\setminus\tilde K$, by the vanishing theorem \cite[Corollary 5.13]{BrShi17}, the second summand in the right hand side of \eqref{E:tilD splits} vanishes. The first summand in the right hand side of \eqref{E:tilD splits} is given by the Atiyah--Patodi--Singer index theorem \cite[Theorem 3.10]{APS1} (Note that  $\Sigma$ is outside of an essential support of $\hat\DD^+$ and, hence, the restriction of $\hat\DD^+$ to $\Sigma$ is invertible. Hence, the kernel of the restriction of $\hat\DD^+$ to $\Sigma$ is trivial) 
\[
	\ind\hat\DD^+_{\tilde K}\;=\;
	\int_{\tilde K}\alpha_{\rm AS}(\hat\DD^+)\;-\;\frac{1}{2}\eta(0),
\]
where $\eta(0)$ is the $\eta$-invariant of the restriction of $\hat\DD^+$ to $\Sigma$. 

As $\AS(\hat\DD^+)= \AS(\tilde\DD^+)\equiv 0$ on $W$ and $\hat\DD^+\equiv\tilde\DD^+$ elsewhere, we have  
\[
	\int_{\tilde K}\alpha_{\rm AS}(\hat\DD^+) \ = \ 
	\int_{\tilde K}\alpha_{\rm AS}(\tilde\DD^+) \ = \ 
	\int_{\tilM}\alpha_{\rm AS}(\tilde\DD^+).
\]
To finish the proof of the lemma it suffices now to show $\eta(0)=0$.

Let $\omega$ be the inward (with respect to $\tilde K$) unit normal one-form along $\Sigma$. Recall that $\Sigma=S^1\times L$. We denote the coordinate along $S^1$ by $\theta$. Suppose that $\{\omega,d\theta,e_1,\cdots,e_m\}$ forms a local orthonormal frame of $T^*\tilM$ on $\Sigma$. Then the restriction of $\hat\DD^+=\hat D^++\hat\Psi^+$ to $\Sigma$ can be written as
\[
\hat\AAA^+_\Sigma\;=\;-\sum_{i=1}^m\tilde c(\omega)\tilde c(e_i)\nabla_{e_i}^{\tilde E}-\tilde c(\omega)\tilde c(d\theta)\partial_\theta-\tilde c(\omega)\hat\Psi^+
\]
which maps $C^\infty(\Sigma,\tilde E^+|_\Sigma)$ to itself. We define a unitary isomorphism $\Theta$ on the space $C^\infty(\Sigma,\tilde E|_\Sigma)$ given by
\[
\Theta u(\theta,y)\;:=\;-\tilde c(\omega)\tilde c(d\theta)u(-\theta,y).
\]
One can check that $\Theta$ anticommutes with $\hat\AAA^+_\Sigma$. As a result, the spectrum of $\hat\AAA^+_\Sigma$ is symmetric about 0. Therefore $\eta(0)=0$ and lemma is proved.
\end{proof}

\subsection{Proof of Theorem~\ref{T:indep of D}}\label{SS:prindep of D}

Recall that we denote by $B_0$ and $B_1$ the APS boundary conditions for $\DD^+=\tilde{\DD}^+|_{M}$ at $N_0$ and $N_1$ respectively. Let $(\DD'')^{+}$ denote the restriction of $\tilde\DD^+$ to $-M'=\tilde{M}\setminus M$. Let $\oB_0$ and $\oB_1$ be the dual APS boundary conditions for $(\DD'')^{+}$ at $N_0$ and $N_1$ respectively. By the Splitting Theorem~\ref{T:splitting},
\begin{equation}\label{E:tilD=D+D'}\notag
	\ind\tilde{\DD}^+ \ = \ 
	\ind \DD^+_{B_0\oplus B_1}\ + \ \ind (\DD'')^{+}_{\oB_0\oplus \oB_1}.
\end{equation}
By Lemma~\ref{L:indtildeD}, we obtain 
\[
	\ind \DD^+_{B_0\oplus B_1}\ + \ \ind (\DD'')^{+}_{\oB_0\oplus \oB_1}
	\ =\ \int_M\alpha_{\rm AS}(\DD^+)\ +\ 
	\int_{M'}\alpha_{\rm AS}((\DD'')^{+}),
\]
which means
\begin{equation}\label{E:D=-D''}
	\ind \DD^+_{B_0\oplus B_1}\ - \ \int_M\alpha_{\rm AS}(\DD^+)
	\ =\ 
	-\ind (\DD'')^{+}_{\oB_0\oplus \oB_1}\ +\ 
	\int_{M'}\alpha_{\rm AS}((\DD'')^{+}).
\end{equation}

By Lemma~\ref{L:tildeD}, $(\DD'')^+= -(\DD')^-$. Thus $\oB_0\oplus\oB_1$ is the adjoint of the APS boundary condition for $(-\DD')^+$ (cf. Definition \ref{D:gAPS}). Therefore,
\[
	\ind (\DD'')^{+}_{\oB_0\oplus \oB_1}
	\ = \ \ind (-\DD')^-_{\oB_0\oplus \oB_1}
	\ = \ -\ind (-\DD')^{+}_{B_0\oplus B_1} \ = \ 
	-\, \ind (\DD')^{+}_{B_0\oplus B_1},
\]
where we used \eqref{E:indD-indD*} in the middle equality. Also by the construction of local index density,
\[
	\alpha_{\rm AS}((\DD'')^{+})\ =\ \alpha_{\rm AS}((-\DD')^-)
	\ =\ 
	\alpha_{\rm AS}((\DD')^{-})
	\ = \ -\alpha_{\rm AS}((\DD')^{+}).
\]
Combining these equalities with \eqref{E:D=-D''} we obtain \eqref{E:indep of D}. \hfill$\square$

\section{The relative $\eta$-invariant}\label{S:releta}

In the previous section we proved that  on an essentially cylindrical manifold $M$ the difference $\ind \DD_{B_0\oplus B_1}- \int_M \AS(\DD)$  depends only on the restriction of $\DD$ to the boundary, i.e., on the  operators $\AAA_0$ and $-\AAA_1$. In this section we use this fact to define the {\em relative $\eta$-invariant} $\eta(\AAA_1,\AAA_0)$ and show that it has properties similar to the difference of $\eta$-invariants $\eta(\AAA_1)-\eta(\AAA_0)$ of operators on compact manifolds. For special cases, \cite{FoxHaskell05}, when the index can be computed using heat kernel asymptotics, we show that  $\eta(\AAA_1,\AAA_0)$ is indeed equal to the difference of the $\eta$-invariants of $\AAA_1$ and $\AAA_0$. In the next section we discuss the connection between the relative $\eta$-invariant and the spectral flow.

In the case when $\AAA_0, \AAA_1$ are operators on even-dimensional manifolds, an analogous construction  was proposed in \cite[\S8]{BrShi17}. Even though the definition of the relative $\eta$-invariant  for operators on odd-dimensional manifolds proposed in this section is slightly more involved than the definition in \cite{BrShi17}, we show that most of the properties of $\eta(\AAA_1,\AAA_0)$ remain the same. 

\subsection{Almost compact cobordisms}\label{SS:almost compact cobordism}
Let $N_0$ and $N_1$ be two complete {\em odd-dimensional} Riemannian manifolds and let $\AAA_0$ and $\AAA_1$ be self-adjoint strongly Callias-type operators on $N_0$ and $N_1$ respectively, cf. Definition~\ref{D:saCallias}.

\begin{definition}\label{D:almost compact cobordism}
An {\em almost compact cobordism} between $\AAA_0$ and $\AAA_1$ is a pair $(M,\DD)$, where $M$ is an essentially cylindrical manifold with $\pM=N_0\sqcup N_1$ and $\DD$ is a graded self-adjoint strongly Callias-type operator on $M$ such that 
\begin{enumerate}
\item $M$ is an almost compact essential support of $\DD$;
\item $\DD$ is product near $\pM$;
\item 	The restriction of $\DD^+$ to $N_0$ is equal to $\AAA_0$ and the restriction of $\DD^+$ to $N_1$ is equal to $-\AAA_1$.
\end{enumerate}
If there exists an almost compact cobordism between $\AAA_0$ and $\AAA_1$ we say that operator $\AAA_0$ is {\em cobordant} to operator $\AAA_1$.
\end{definition}

\begin{lemma}\label{L:antisymmetry}
An almost compact cobordism is an equivalence relation on the set of self-adjoint strongly Callias-type operators, i.e.,
\begin{enumerate}
\item \ If\/ $\AAA_0$ is cobordant to $\AAA_1$ then $\AAA_1$ is cobordant to $\AAA_0$. 

\item \ Let $\AAA_0,\AAA_1$ and $\AAA_2$ be self-adjoint strongly Callias-type operators on odd-dimensional complete Riemannian manifolds $N_0, N_1$ and $N_2$ respectively. 
Suppose $\AAA_0$ is cobordant to $\AAA_1$ and $\AAA_1$ is cobordant to $\AAA_2$. Then $\AAA_0$ is cobordant to $\AAA_2$.
\end{enumerate}
\end{lemma}

\begin{proof}
The proof is a verbatim repetition of the proof of Lemmas~8.2 and 8.3 of \cite{BrShi17}. 
\end{proof}

\begin{definition}\label{D:releta}
Suppose $\AAA_0$ and $\AAA_1$ are cobordant self-adjoint strongly Callias-type operators and let $(M,\DD)$ be an almost compact cobordism between them. Let $B_0= H_{(-\infty,0)}^{1/2}(\AAA_0)$ and $B_1= H_{(-\infty,0)}^{1/2}(-\AAA_1)$ be the APS boundary conditions for $\DD^+$. The {\em relative $\eta$-invariant} is defined as 
\begin{equation}\label{E:releta}
	\eta(\AAA_1,\AAA_0) \ = \  2\,\left(\,\ind \DD^+_{B_0\oplus B_1}
	- \int_M\,\AS(\DD^+)\,\right)
	 \ + \ \dim\ker \AAA_0\ + \ \dim\ker \AAA_1.
\end{equation}
\end{definition}
Theorem~\ref{T:indep of D}  implies that  $\eta(\AAA_1,\AAA_0)$ is independent of the choice of the cobordism $(M,\DD)$. 

\begin{remark}\label{R:rel eta}
Sometimes it is  convenient to use the dual APS boundary conditions $\oB_0= H_{(-\infty,0]}^{1/2}(\AAA_0)$ and $\oB_2= H_{(-\infty,0]}^{1/2}(-\AAA_1)$ instead of $B_0$ and $B_1$. It follows from Corollary \ref{C:change of B} that the relative $\eta$-invariant can be written as
\begin{equation}\label{E:B1'B2'}
		\eta(\AAA_1,\AAA_0) \ = \  
		2\,\left(\,\ind \DD^+_{\oB_0\oplus \oB_1}- \int_M\,\AS(\DD^+)\,\right)
	 \ - \ \dim\ker \AAA_0\ - \ \dim\ker \AAA_1.
\end{equation}
\end{remark}

\subsection{The case when the heat kernel has an asymptotic expansion}\label{SS:FoxHaskell}
In \cite{FoxHaskell05}, Fox and Haskell studied the index of a boundary value problem on manifolds of bounded geometry. They showed that under certain conditions (satisfied for natural operators on manifolds with conical or cylindrical ends) on $M$ and $\DD$,  the heat kernel $e^{-t(\DD_B)^*\DD_B}$ is of trace class and its trace has an asymptotic expansion similar to the one on compact manifolds. In this case the $\eta$-function, defined by a usual formula 
\[
		\eta(s;\AAA) \ := \ \sum_{\lambda\in{\rm spec}(\AAA)} 
		{\rm sign}(\lambda)\,|\lambda|^s, 
		\qquad \operatorname{Re} s \ll 0,
\] 
is an analytic function of $s$, which has a meromorphic continuation to the whole complex plane and is regular at 0. So one can define the $\eta$-invariant of $\AAA$ by $\eta(\AAA)= \eta(0;\AAA)$.

\begin{proposition}\label{P:FoxHaskell}
Suppose now that $\DD$ is an operator on an essentially cylindrical manifold $M$ which satisfies the conditions of\, \cite{FoxHaskell05}. We also assume that $\DD$ is product near $\pM=N_0\sqcup N_1$ and that $M$ is an almost compact essential support of $\DD$. Let $\AAA_0$ and $-\AAA_1$ be the restrictions of $\DD^+$ to $N_0$ and $N_1$ respectively. Let $\eta(\AAA_j)$ ($j=0,1$) be the $\eta$-invariant of $\AAA_j$.   Then 
\begin{equation}\label{E:FoxHaskell}
	\eta(\AAA_1,\AAA_0) \ = \ \eta(\AAA_1)\ - \ \eta(\AAA_0).
\end{equation}
\end{proposition}

\begin{proof}
An analogue of this proposition for the case when $\dim M$ is odd is proven in \cite[Proposition~8.8]{BrShi17}. This proof extends to the case when $\dim M$ is even without any changes. 
\end{proof}

\subsection{Basic properties of the relative $\eta$-invariant}\label{SS:properties eta}

Proposition~\ref{P:FoxHaskell} shows that under certain conditions the $\eta$-invariants of $\AAA_0$ and $\AAA_1$ are defined and $\eta(\AAA_1,\AAA_0)$ is their difference. We now show that in general case, when $\eta(\AAA_0)$ and $\eta(\AAA_1)$ do not necessarily exist,  $\eta(\AAA_1,\AAA_0)$ behaves like it were a difference of an invariant of $N_1$ and an invariant of $N_0$.

\begin{proposition}[Antisymmetry]\label{P:antisymmetry eta}	
Suppose $\AAA_0$ and $\AAA_1$ are cobordant self-adjoint strongly Callias-type operators. Then 
\begin{equation}\label{E:antisymmetry}
	\eta(\AAA_0,\AAA_1)\ = \ -\,\eta(\AAA_1,\AAA_0).
\end{equation}
\end{proposition}

\begin{proof}
Let $-M$ denote the manifold $M$ with the opposite orientation and let $\tilde{M}:= M\cup_\pM(-M)$ denote the {\em double} of $M$. Let $\DD$ be an almost compact cobordism between $\AAA_0$ and $\AAA_1$. Using the construction of  Section~\ref{SS:glue} (with $\DD'= \DD$) we obtain a graded self-adjoint strongly Callias-type operator $\tilde{\DD}$ on $\tilde{M}$ whose restriction to $M$ is isometric to $\DD$. Let $\DD''$ denote the restriction of $\tilde{\DD}$ to $-M=\tilde{M}\setminus{}M$. Then the restriction of $(\DD'')^+$ to $N_1$ is equal to $\AAA_1$ and the restriction of $(\DD'')^+$ to $N_0$ is equal to $-\AAA_0$.

Let   
\[
	\begin{aligned}
	\oB_0 \ &= \ H^{1/2}_{[0,\infty)}(\AAA_0)\ = \  H^{1/2}_{(-\infty,0]}(-\AAA_0),\\
	\oB_1 \ &= \  H^{1/2}_{[0,\infty)}(-\AAA_1) \ = \ H^{1/2}_{(-\infty,0]}(\AAA_1)
	\end{aligned}
\] 
be the dual APS boundary conditions for $(\DD'')^+$.  By \eqref{E:D=-D''}, 
\begin{equation}\label{E:indD''=indD}
			 \ind (\DD'')^+_{\oB_0\oplus \oB_1}\ -\ 
			\int_{M'}\alpha_{\rm AS}((\DD'')^+)
			\ = \
			-\ind \DD^+_{B_0\oplus B_1}\ + \ \int_M\alpha_{\rm AS}(\DD^+).
\end{equation}

Since $\DD''$ is an almost compact cobordism between $\AAA_1$ and $\AAA_0$ we conclude  from \eqref{E:B1'B2'} that
\begin{equation}\label{E:etaA0A1}
	\eta(\AAA_0,\AAA_1)\ = \ 
	2\,\left(\,
	  \ind (\DD'')^+_{\oB_0\oplus \oB_1}- 
	  	\int_{M'}\alpha_{\rm AS}((\DD'')^+)\,\right)
	\ - \ \dim\ker \AAA_0\ - \ \dim\ker \AAA_1.
\end{equation}
Combining  \eqref{E:etaA0A1} and \eqref{E:indD''=indD} we obtain \eqref{E:antisymmetry}.
\end{proof}

Note that \eqref{E:antisymmetry} implies that 
\begin{equation}\label{E:etaAA}
	\eta(\AAA,\AAA)\ = \ 0
\end{equation}
for every self-adjoint strongly Callias-type operator $\AAA$. 

\begin{proposition}[The cocycle condition]\label{P:cocycle}
Let $\AAA_0,\AAA_1$ and $\AAA_2$ be self-adjoint strongly Callias-type operators which are cobordant to each other. Then 
\begin{equation}\label{E:cocycle}
	\eta(\AAA_2,\AAA_0)\ = \ 
	\eta(\AAA_2,\AAA_1)\ + \ \eta(\AAA_1,\AAA_0).
\end{equation}
\end{proposition}
\begin{proof}
Let $M_1$ and $M_2$ be essentially  cylindrical manifolds such that $\pM_1= N_0\sqcup N_1$ and $\pM_2= N_1\sqcup N_2$. Let $\DD_1$ be an operator on $M_1$ which is an almost compact cobordism between $\AAA_0$ and $\AAA_1$. Let $\DD_2$ be an operator on $M_2$ which is an almost compact cobordism between $\AAA_1$ and $\AAA_2$. Then the operator $\DD_3$ on $M_1\cup_{N_1} M_2$ whose restriction to $M_j$ ($j=1,2$) is equal to $\DD_j$ is an almost compact cobordism between $\AAA_0$ and $\AAA_2$. 

Let $B_0$ and $B_1$ be the APS boundary conditions for $\DD_1^+$ at $N_0$ and $N_1$ respectively. Then $\oB_1= H^{1/2}_{[0,\infty)}(\AAA_1)$ is equal to the dual APS boundary condition for $\DD_2^+$. Let $B_2$ be the APS boundary condition for $\DD_2^+$ at $N_2$. From Corollary \ref{C:change of B} we obtain 
\begin{equation}\label{E:indBadB}
			\eta(\AAA_2,\AAA_1) \ = \  
		2\,\left(\,\ind (\DD_2^+)_{\oB_1\oplus B_2}- \int_M\,\AS(\DD_2^+)\,\right)
	 \ - \ \dim\ker \AAA_1\ + \ \dim\ker \AAA_2.
\end{equation}
By the Splitting Theorem~\ref{T:splitting}
\begin{equation}\label{E:splitB0B1B3}
	\ind (\DD_3^+)_{B_0\oplus B_2}\ = \ 
	\ind (\DD_1^+)_{B_0\oplus B_1} \ + \ \ind (\DD_2^+)_{\oB_1\oplus B_2}.
\end{equation}
Clearly, 
\begin{equation}\label{E:split int}
		\int_{M_1\cup M_2}\AS(\DD_3^+)\ = \ 
	\int_{M_1}\,\AS(\DD_1^+)\ + \ \int_{M_2}\,\AS(\DD_2^+).
\end{equation}
Combining \eqref{E:indBadB}, \eqref{E:splitB0B1B3}, and \eqref{E:split int} we obtain \eqref{E:cocycle}. 
\end{proof}

\section{The spectral flow}\label{S:sp flow}

Suppose $\AA:= \{\AAA^s\}_{0\le s\le 1}$ is a smooth family of self-adjoint elliptic operators on a closed  manifold $N$.  Let $\oeta(\AAA^s)\in \RR/\ZZ$ denote the mod $\ZZ$ reduction of the $\eta$-invariant $\eta(\AAA^s)$.  Atiyah, Patodi, and Singer, \cite{APS3}, showed that $s\mapsto \oeta(\AAA^s)$ is a smooth function whose derivative $\frac{d}{ds}\oeta(\AAA^s)$ is given by an explicit local formula. 
Further, Atiyah, Patodi and Singer, \cite{APS3}, introduced a notion of spectral flow $\spf(\AA)$ and showed that it  computes the net number of integer jumps of $\eta(\AAA^s)$, i.e., 
\[
	2\,\spf(\AA)\ = \ 
	\eta(\AAA^1)\ - \ \eta(\AAA^0)\ - \ 
	\int_0^1\,\big(\frac{d}{ds}\oeta(\AAA^s)\,\big)\, ds. 
\]
In this section we consider a family of self-adjoint strongly Callias-type operators $\AA= \{\AAA^s\}_{0\le s\le 1}$ on a complete  Riemannian manifold. Assuming that the restriction of $\AAA^s$ to a complement of a compact set $K\subset N$ is independent of  $s$, we show that for any operator $\AAA_0$ cobordant to $\AAA^0$  the  mod $\ZZ$  reduction $\oeta(\AAA^s,\AAA_0)$ of the relative $\eta$-invariant depends smoothly on $s$ and 
\[
	2\,\spf(\AA)\ = \ 
	\eta(\AAA^1,\AAA_0)- \eta(\AAA^0,\AAA_0)\ - \ 
	\int_0^1\,\big(\frac{d}{ds}\oeta(\AAA^s,\AAA_0)\,\big)\, ds. 
\]

\subsection{A family of boundary operators}\label{SS:familyAA}
Let $E_N\to N$ be a Dirac bundle over a complete {\em odd-dimensional} Riemannian manifold $N$. We denote the Clifford multiplication of $T^*N$ on $E_N$ by $c_N:T^*N\to \End(E_N)$. Let $\AA= \{\AAA^s\}_{0\le s\le 1}$ be a family of self-adjoint strongly Callias-type operators $\AAA^s:C^\infty(N,E_N) \to C^\infty(N,E_N)$.

\begin{definition}\label{D:almost constant}
The family $\AA= \{\AAA^s\}_{0\le s\le 1}$ is called {\em almost constant} if there exists a compact set $K\subset N$ such that the restriction of $\AAA^s$ to $N\setminus K$ is independent of $s$. 
\end{definition}

Consider the cylinder $M:=[0,1]\times N$  and denote by $t$ the coordinate along $[0,1]$. Set 
\[
	E^+= E^-\ := \ [0,1]\times E_N.
\]
Then $E=E^+\oplus E^-\to M$ is naturally a $\ZZ_2$-graded Dirac bundle over $M$ with 
\[
	c(dt) \ := \ \begin{pmatrix}
	0& -\id_{E_N}\\\id_{E_N}&0
	\end{pmatrix} 
\]
and 
\[
	c(\xi)\ : = \ \begin{pmatrix}
	0& c_N(\xi)\\ c_N(\xi)&0
	\end{pmatrix}, \qquad\text{for}\quad \xi\in T^*N.
\]

\begin{definition}\label{D:smoothD}
The family $\AA= \{\AAA^s\}_{0\le s\le 1}$ is called {\em smooth} if  
\begin{equation}\label{E:smoothD}
	\DD\ : = \ c(dt)\,\left(\, \p_t+
		\begin{pmatrix}
		\AAA^t&0\\0&-\AAA^t
		\end{pmatrix}\,\right):\ 
	C^\infty(M,E)\to C^\infty(M,E)
\end{equation}
is a smooth differential operator on $M$. 
\end{definition}

Fix a smooth non-decreasing function $\kappa:[0,1]\to [0,1]$ such that $\kappa(t)=0$ for $t\le1/3$ and $\kappa(t)=1$ for $t\ge2/3$ and consider the operator 
\begin{equation}\label{E:interpolationD}
	\DD\ : = \ c(dt)\,\left(\, \p_t+
		\begin{pmatrix}
		\AAA^{\kappa(t)}&0\\0&-\AAA^{\kappa(t)}
		\end{pmatrix}\,\right):\ C^\infty(M,E)\to C^\infty(M,E).
\end{equation}
Then $\DD$ is product near $\pM$. If $\AA$ is a smooth almost constant family of self-adjoint strongly Callias-type operators then \eqref{E:interpolationD} is a strongly Callias-type operator for which $M$ is an almost compact essential support. Hence it is a non-compact cobordism (cf. Definition~\ref{D:almost compact cobordism}) between $\AAA^0$ and $\AAA^1$. 

\subsection{The spectral section}\label{SS:sp section}

If\/ $\AA= \{\AAA^s\}_{0\le s\le 1}$ is a smooth almost constant family of self-adjoint strongly Callias-type operators then it satisfies the conditions of the Kato Selection Theorem \cite[Theorems~II.5.4 and II.6.8]{Kato95book}, \cite[Theorem~3.2]{Nicolaescu95}. Thus there is a family of eigenvalues $\lambda_j(s)$ ($j\in \ZZ$)  which depend continuously on $s$. We order the eigenvalues so that $\lambda_j(0)\le \lambda_{j+1}(0)$ for all $j\in \ZZ$ and $\lambda_j(0)\le 0$ for $j\le 0$ while $\lambda_j(0)>0$ for $j>0$. 

Atiyah, Patodi and Singer \cite{APS3} defined the spectral flow $\spf(\AA)$ for a family of operators satisfying the conditions of the Kato Selection Theorem  (\cite[Theorems~II.5.4 and II.6.8]{Kato95book}, \cite[Theorem~3.2]{Nicolaescu95}) as an integer that counts the net number of eigenvalues that change sign when $s$ changes from 0 to 1. 
Several other equivalent definitions of the spectral flow based on different assumptions on the family $\AA$ exist in the literature. For our purposes the most convenient is the  Dai and Zhang's definition \cite{DaiZhang98} which is based on the notion of {\em spectral section} introduced by Melrose and Piazza \cite{MelrosePiazza97}.

\begin{definition}\label{D:spectral section}
A {\em spectral section} for $\AA$ is a continuous family $\PP = \{P^s\}_{0\le s\le 1}$ of self-adjoint projections such that there exists a constant $R>0$ such that for all $0\le s\le 1$, if $\AAA^su= \lambda u$ then
\[
	P^su \ = \ \begin{cases} 
		0, \quad &\text{if}\quad \lambda< -R;\\ 
		u, \quad	 &\text{if}\quad \lambda> R.	
		\end{cases}
\]
\end{definition}

If $\AA$ satisfies the conditions of the Kato Selection Theorem, then the arguments of the proof of  \cite[Proposition~1]{MelrosePiazza97} show that $\AA$ admits a spectral section.

\subsection{The spectral flow}\label{SS:sp flow}
Let $\PP = \{P^s\}$ be a spectral section for $\AA$. Set $B^s:= \ker P^s$. Let $B_0^s:= H^{1/2}_{(-\infty,0)}(\AAA^s)$ denote the APS boundary condition defined by the boundary operator $\AAA^s$.   Since the spectrum of $\AAA^s$ is discrete, it follows immediately from the definition of the spectral section that for every $s\in[0,1]$ the space $B^s$ is a finite rank perturbation of $B_0^s$, cf. Section~\ref{SS:gen bc}. Recall that  the relative index $[B^s,B_0^s]$ was defined in Definition~\ref{D:rel index}. Following Dai and Zhang \cite{DaiZhang98} (see also \cite[\S9.8]{BrShi17}) we give the following definition.

\begin{definition}\label{D:sp flow}
Let $\AA= \{\AAA^s\}_{0\le s\le 1}$ be a smooth almost constant  family of self-adjoint strongly Callias-type operators which admits a spectral section $\PP = \{P^s\}_{0\le s\le 1}$. Assume that the operators $\AAA^0$ and $\AAA^1$ are invertible. Let $B^s:= \ker P^s$ and $B_0^s:= H^{1/2}_{(-\infty,0)}(\AAA^s)$. The {\em spectral flow}\/ $\spf(\AA)$ of the family $\AA$ is defined by the formula
\begin{equation}\label{E:sp flow=sp section}
	\spf(\AA)\ := \ [B^1,B_0^1] \ - \ [B^0,B^0_0].
\end{equation}
\end{definition}
By Theorem~1.4 of \cite{DaiZhang98} the spectral flow is independent of the choice of the spectral section $\PP$ and computes the net number of eigenvalues that change sign when $s$ changes from 0 to 1.

\begin{lemma}\label{L:spflow-A}
Let $-\AA$ denote the family $\{-\AAA^s\}_{0\le s\le 1}$. Then 
\begin{equation}\label{E:spflow-A}
	\spf(-\AA)\  = \ -\spf(\AA).
\end{equation}
\end{lemma}
\begin{proof}
The lemma is an immediate consequence of Lemma~5.7 of \cite{BrShi17}.
\end{proof}

\subsection{Deformation of the relative $\eta$-invariant}\label{SS:def eta}

Let $\AA= \{\AAA^s\}_{0\le s\le 1}$ be a smooth almost constant family of self-adjoint strongly Callias-type operators on a complete odd-dimensional Riemannian manifold $N_1$. Let $\AAA_0$ be another self-adjoint strongly Callias-type operator, which is cobordant to $\AAA^0$. In Section~\ref{SS:familyAA} we showed that  $\AAA^0$ is cobordant to $\AAA^s$ for all $s\in [0,1]$. Hence, by Lemma~\ref{L:antisymmetry}.(ii), $\AAA_0$ is cobordant to $\AAA^1$. In this situation we say the $\AAA_0$ is {\em cobordant to the family} $\AA$.

The following theorem is the main result of this  section.

\begin{theorem}\label{T:sp flow}
Suppose 
\(\AA  =  \big\{
	  \AAA^s:\,C^\infty(N_1,E_1)\to C^\infty(N_1,E_1)\big\}_{0\le s\le 1}  
\)
is a smooth almost constant family of self-adjoint strongly Callias-type operators on a complete odd-dimensional Riemannian manifold $N_1$. Assume that $\AAA^0$ and $\AAA^1$ are invertible. Let $\AAA_0:C^\infty(N_0,E_0)\to C^\infty(N_0,E_0)$ be an invertible self-adjoint strongly Callias-type operator on a complete Riemannian manifold $N_0$ which is cobordant to the family $\AA$. Then the   mod $\ZZ$ reduction $\oeta(\AAA^s,\AAA_0)\in \RR/\ZZ$  of the relative $\eta$-invariant depends smoothly on $s\in[0,1]$ and 
\begin{equation}\label{E:sp flow}
	\eta(\AAA^1,\AAA_0)\ - \ \eta(\AAA^0,\AAA_0)
	\ - \ \int_0^1\,\big(\frac{d}{ds}\oeta(\AAA^s,\AAA_0)\big)\,ds
	\ = \ 2\,\spf(\AA).
\end{equation}
\end{theorem}

The proof of this theorem occupies Sections~\ref{SS:DDr}--\ref{SS:prsp flow}.

\subsection{A family of almost compact cobordisms}\label{SS:DDr}

Let $M$ be an essentially cylindrical manifold whose boundary is the disjoint union of $N_0$ and $N_1$. First, we construct a smooth family $\DD^r$ $(0\le r\le1$) of graded self-adjoint strongly Callias-type operators on the manifold
\begin{equation}\label{E:M'}
	M'\ := \ M\cup_{N_1} \big([0,1]\times N_1\,\big), 
\end{equation}
such that for each $r\in [0,1]$ the pair $(M',\DD^r)$ is an almost compact cobordism between $\AAA_0$ and $\AAA^r$.

Let $\DD:C^\infty(M,E)\to C^\infty(M,E)$ be an almost compact cobordism between $\AAA_0$ and $\AAA^0$. Let $E_0$ and $E_1$ denote the restrictions of $E$ to $N_0$ and $N_1$ respectively. 

Let $M'$ be given by \eqref{E:M'} and let $E'\to M'$ be the bundle over $M'$ whose restriction to $M$ is equal to $E$ and whose restriction to the cylinder $[0,1]\times N_1$ is equal to $[0,1]\times E_1$. 

We fix a smooth function $\rho:[0,1]\times [0,1]\to [0,1]$ such that for each $r\in [0,1]$
\begin{itemize}
\item  the function $s\mapsto \rho(r,s)$ is non-decreasing.
\item
$\rho(r,s)=0$ for $s\le1/3$ and $\rho(r,s)= r$ for $s\ge2/3$.
\end{itemize}
Consider the family of strongly Callias-type operators $\DD^r:C^\infty(M',E')\to C^\infty(M',E')$ whose restriction to $M$ is equal to $\DD$ and whose restriction to $[0,1]\times N_1$ is given by 
\[
	\DD^r\ : = \ c(dt)\,\left(\, \p_t+
		\begin{pmatrix}
		\AAA^{\rho(r,t)}&0\\0&-\AAA^{\rho(r,t)}
		\end{pmatrix}\,\right).
\]
Then $\DD^r$ is an almost compact cobordism between $\AAA_0$ and $\AAA^r$. In particular, the restriction of $\DD^r$ to $N_1$ is equal to $-\AAA^r$. 

Recall that we denote by $-\AA$ the family $\{-\AAA^s\}_{0\le s\le 1}$.
Let $\PP= \{P^s\}$ be a spectral section for $-\AA$. Then for each $r\in [0,1]$ the space  $B^r:= \ker P^r$ is a finite rank perturbation of the APS boundary condition  for $\DD^r$ at $\{1\}\times N_1$. Let $B_0:= H^{1/2}_{(-\infty,0)}(\AAA_0)$ be the APS boundary condition for $\DD^r$ at $N_0$. Then, by Proposition~\ref{P:change of B}, the operator  $\DD^r_{B_0\oplus B^r}$ is Fredholm. Recall that the domain $\dom \DD^r_{B_0\oplus B^r}$ consists of sections $u$ whose restriction to $\pM'=N_0\sqcup N_1$ lies in $B_0\oplus B^r$.

\begin{lemma}\label{L:indDr}
$\ind \DD^r_{B_0\oplus B^r}= \ind \DD^1_{B_0\oplus B^1}$ for all $r\in [0,1]$. 
\end{lemma}
\begin{proof}
For $r_0,r\in [0,1]$, let $\pi_{r_0r}:\,B^{r_0} \to B^r$
denote the orthogonal projection. Then for every $r_0\in[0,1]$ there exists $\varepsilon>0$ such that if $|r-r_0|<\varepsilon$ then $\pi_{r_0r}$ is an isomorphism. As in the proofs of \cite[Theorem 5.11]{BrShi17} and \cite[Theorem 8.12]{BaerBallmann12}, it induces an isomorphism
\[
	\Pi_{r_0r}\,:\,\dom\DD^{r_0}_{B_0\oplus B^{r_0}} \ \to \ 
		\dom\DD^{r}_{B_0\oplus B^{r}}.
\]
Hence
\begin{equation}\label{E:DcircPi}
		\ind \big(\DD^{r}_{B_0\oplus B^{r}}\circ \Pi_{r_0r}\big) 
		\ = \ \ind \DD^{r}_{B_0\oplus B^{r}}.
\end{equation}
Since for $|r-r_0|<\varepsilon$
\[
	\DD^{r}_{B_0\oplus B^{r}}\circ \Pi_{r_0r}:\, \dom\DD^{r_0}_{B_0\oplus B^{r_0}}
	\ \to \ L^2(M',E')
\]
is a continuous family of bounded operators, $\ind \DD^{r}_{B_0\oplus B^{r}}\circ \Pi_{r_0r}$ is independent of $r$. The lemma follows now from \eqref{E:DcircPi}.
\end{proof}

\subsection{Variation of the reduced relative $\eta$-invariant}\label{SS:reducedeta}	

By Definition~\ref{D:releta}, the mod $\ZZ$ reduction of the relative $\eta$-invariant is given by 
\begin{equation}\label{E:redeta formula}
	\oeta(\AAA^r,\AAA_0)\ := \ -2\,\int_{M'}\, \AS(\DD^r).
\end{equation} 
It follows that $\oeta(\AAA^r,\AAA_0)$ depends smoothly on $r$ and 
\begin{equation}\label{E:ddr oeta}
	\frac{d}{dr}\oeta(\AAA^r,\AAA_0)\ = \ -2\,\int_{M'}\, \frac{d}{dr} \AS(\DD^r).
\end{equation}
A more explicit local expression for the right hand side of this equation is given in Section~\ref{SS:ddtoeta}. For the moment we just note that \eqref{E:ddr oeta} implies that
\begin{equation}\label{E:intddr oeta}
	\int_0^1\,\big(\frac{d}{ds}\oeta(\AAA^s,\AAA_0)\big)\,ds \ = \ 
	-2\,\int_{M'}\, \Big(\AS(\DD^1)- \AS(\DD^0)\,\Big).
\end{equation}

\subsection{Proof of Theorem~\ref{T:sp flow}}\label{SS:prsp flow}
 Since the operators $\AAA_0, \AAA^0$, and $\AAA^1$ are invertible, we have 
\begin{equation}\label{E:releta01}\notag
	\eta(\AAA^j,\AAA_0)= 2\,\left(\, \ind \DD^{j}_{B_0\oplus B^j_0}\  - \ 
	\int_{M'}\,\AS(\DD^j)
	\,\right), \qquad j=0,1.
\end{equation}
Thus, using \eqref{E:intddr oeta}, we obtain
\begin{equation}\label{E:dif eta}
	\eta(\AAA^1,\AAA_0)- \eta(\AAA^0,\AAA_0)\ - \ 
	\int_0^1\,\big(\frac{d}{ds}\oeta(\AAA^s,\AAA_0)\,\big)\, ds 
 	\ = \ 
 	2\,\left(\, \ind \DD^{1}_{B_0\oplus B^1_0}\ - \ 
 	 \ind \DD^{0}_{B_0\oplus B^0_0}\,\right).
\end{equation}

Recall that, by Proposition~\ref{P:change of B},
\begin{equation}\label{E:B0Brind}\notag
	\ind \DD^{r}_{B_0\oplus B^r} \  = \ \ind \DD^{r}_{B_0\oplus B^r_0} 
	  \ + \ [B^r,B^r_0],
\end{equation}
where $B^r= \ker P^r$ and $B^r_0= H^{1/2}_{(-\infty,0)}(\AAA^r)$ are defined in Subsection \ref{SS:sp flow}.
Hence, from \eqref{E:dif eta} we obtain
\begin{multline}\notag
	\frac12\left(\,
	\eta(\AAA^1,\AAA_0)- \eta(\AAA^0,\AAA_0)\ - \ 
	\int_0^1\,\big(\frac{d}{ds}\oeta(\AAA^s,\AAA_0)\,\big)\, ds \,\right)
	\\ = \ 
	\left(\,\ind \DD^1_{B_0\oplus B^1}- [B^1,B^1_0]\,\right) \ -\ 
	\left(\, \ind \DD^0_{B_0\oplus B^0}- [B^0,B^0_0]\,\right)
	\\ \overset{\text{Lemma~\ref{L:indDr} }}{=} \
	-[B^1,B^1_0]\ + \ [B^0,B^0_0] \ = \ -\spf(-\AA) 
	\ \overset{\text{Lemma~\ref{L:spflow-A}}}= \ \spf(\AA).
\end{multline}
\hfill$\square$

\subsection{A local formula for variation of the reduced relative $\eta$-invariant}\label{SS:ddtoeta}
It is well known that there exists a family of  differential forms $\beta_r$ ($0\le r\le 1$), called the {\em transgression form} such that 
\begin{equation}\label{E:transgression}
	d\beta_r\ = \ \frac{d}{dr} \AS(\DD^r).
\end{equation}
The transgression form depends on the symbol of $\DD^r$ and its derivatives with respect to $r$. For geometric Dirac operators one can write very explicit formulas for $\beta_r$. For example, if $\DD^r$ is the signature operator (so that $\AAA^r$ is the odd signature operator) corresponding to a family $\n^r$ of flat connections on $E$,  then $\beta_r= L(M)\wedge \frac{d}{dr}\n^r$, where $L(M)$ is the  $L$-genus of $M$, cf, for example, \cite[Theorem~2.3]{BrKappelerRAT}. For general Dirac-type operators, a formula for $\beta_r$ is more complicated, cf. \cite[\S6]{BrMaschler17}.

We note that since the family $\AAA^r$ is constant outside of the compact set $K$, the form $\beta_r$ vanishes outside of $K$. Hence, $\int_{\pM'}\beta_r$ is well defined and finite. Thus we obtain from \eqref{E:ddr oeta} that 
\begin{equation}\label{E:ddr local}
	\frac{d}{dr}\oeta(\AAA^r,\AAA_0) \ = \ -2\int_{M'} d\beta_r\ = \ 
	-2\int_{\pM'}\,\beta_r \ = \
	2\;\Big(\int_{\{1\}\times N_1}\beta_r \ - \ \int_{N_0}\,\beta_r\Big).
\end{equation}
Hence, \eqref{E:sp flow} expresses $\eta(\AAA^1,\AAA_0)- \eta(\AAA^0,\AAA_0)$ as a sum of $2\spf(\AA)$ and a local differential geometric expression $2\int_{\pM'}(\int_0^1\beta_r)\,dr$.

\bibliographystyle{amsplain}

\end{document}